\documentclass{article}

\usepackage{titling}
\usepackage{amsfonts}
\usepackage{amsmath}
\usepackage{amsthm}
\usepackage{graphicx}
\usepackage{color}
\usepackage{hyperref}
\usepackage{cite}
\usepackage{refcheck}

\newtheorem{theorem}{Theorem}
\newtheorem{remark}{Remark}

\title{Fourier Beyond Dispersion: Wavenumber Explicit and Precise Accuracy of FDMs for the Helmholtz
  Equation}
\author{Hui Zhang$^*$\\
  \small{School of Mathematics and
    Physics, Xi'an Jiaotong-Liverpool University, Ren'ai Rd. 111, Suzhou 215123, China}\\
  \small{Email: hui.zhang@xjtlu.edu.cn}}

\date{\small{\today}}

\begin{document}
\maketitle

\begin{abstract}
  We propose a practical tool for evaluating and comparing the accuracy of FDMs for the Helmholtz
  equation. The tool based on Fourier analysis makes it easy to find wavenumber explicit order of
  convergence, and can be used for rigorous proof. It fills in the gap between the dispersion
  analysis and the actual error with source term.
\end{abstract}



\section{Introduction.}

Discretization of the Helmholtz equation suffers from the `pollution effect':``the accuracy of the
numerical solution deteriorates with increasing non-dimensional wave number
$k$''\cite{deraemaeker1999dispersion}. It is related to that the numerical solution ``has a
wavelength that is different from the exact one''\cite{deraemaeker1999dispersion} called
`dispersion'. Albeit useful \cite{stolk2016dispersion,wu2017dispersion,cocquet2021closed},
dispersion analysis \cite{deraemaeker1999dispersion} has not a trivial path to the total error
estimate \cite{babuska1997pollution} with source term. In particular, the wavenumber explicit
estimate is largely missing in the FDM literature \cite{fu2008compact, wang2014pollution}. Inspired
by \cite{dwarka2021pollution}, we propose a Fourier analysis approach to precise error
estimation. We shall illustrate the approach in 1D. But it is applicable in 2D and 3D. It also
provides a framework for rigorous proof of sharp error estimates. Those extensions will appear
elsewhere.

\section{Analysis in 1D}

We consider the model problem with homogeneous Dirichlet boundary condition
\begin{equation}\label{helm1d}
  \mathcal{H}u:=u''+k^2u = f\text{ in }\Omega:=(0,1), \quad u(0)=u(1)=0.
\end{equation}
We assume $f\in H^p_0(0,1)$, $p\ge 1$ and $k\not\in\pi\mathbb{Z}$. Then \eqref{helm1d} has a unique
solution $u\in H^{p+2}(0,1)\cap H^1_0(0,1)$. It holds
\[
  f(x)=\sum_{n=1}^{\infty}\hat{f}_n\sin (n\pi x),\quad %
  \hat{f}_n:=2\int_0^1f(x)\sin(n\pi x)~\mathrm{d}x,\quad
  u(x)=\sum_{n=1}^{\infty}\hat{u}_n\sin(n\pi x),\quad %
  \hat{u}_n:=2\int_0^1u(x)\sin (n\pi x)~\mathrm{d}x.
\]
From the equation \eqref{helm1d}, we get
$\hat{u}_n=\frac{1}{k^2-\xi_n^2}\hat{f}_n=:\hat{\mathcal{H}}_n^{-1}\hat{f}_n$ where $\xi_n=n\pi$.

We discretize the domain $(0,1)$ into the uniform grids $\Omega_h:=\{x_j=jh, j=1,\ldots,N-1\}$ with
$h=\frac{1}{N}$ and $N\ge 4$ an integer. Denote a finite difference scheme by
$\mathcal{H}^hu^h=\mathcal{R}^hf=:f^h$ where $u^h$ and $\mathcal{R}^hf$ are functions on $\Omega_h$.
We have the discrete Fourier transform
\[
  f^h(x)=\sum_{n=1}^{N-1}\hat{f}^h_n\sin (n\pi x),\quad %
  \hat{f}^h_n:=\frac{2}{N}\sum_{j=1}^{N-1}f^h(x_j)\sin(n\pi x_j),\quad
\]
\begin{equation}\label{uh}
  u^h(x)=\sum_{n=1}^{N-1}\hat{u}_n^h\sin (n\pi x),\quad %
  \hat{u}^h_n:=\frac{2}{N}\sum_{j=1}^{N-1}u^h(x_j)\sin (n\pi x_j),\quad
\end{equation}
where the formulae of $f^h(x)$ and $u^h(x)$ hold originally for $x\in\Omega_h$ but is also used for
extended definition of $f^h(x)$ and $u^h(x)$ at $x\in\Omega=(0,1)$. From $\mathcal{H}^hu^h=f^h$, we
shall be able to find $\hat{u}^h_n=(\hat{\mathcal{H}}^h_n)^{-1}\hat{f}_n^h$. The error $u-u^h$ is
then
\[
  u-u^h = \sum_{n=1}^{N-1}(\hat{u}_n-\hat{u}_n^h)\sin(n\pi x) +
  \sum_{n=N}^{\infty}\hat{u}_n\sin(n\pi x).
\]
For simplicity, we assume the source $f$ is well resolved on $\Omega_h$ i.e. $f$ is band limited
with $\hat{f}_n=0$ for $n\ge N$. It follows that $f$ and its first and high order derivatives have
the Fourier expansion
\[
  f(x)=\sum_{n=1}^{N-1}\hat{f}_n\sin (n\pi x),\quad %
  f^{(2q)}(x)=\sum_{n=1}^{N-1}(-\xi_n^2)^{q}\hat{f}_n\sin(n\pi x),\quad%
  f^{(2q+1)}(x)=\sum_{n=1}^{N-1}(-1)^q\xi_n^{2q+1}\hat{f}_n\cos(n\pi x), %
  \quad\xi_n=n\pi.
\]
An analysis including under resolved $f$ will appear elsewhere. For band limited $f$, the error
becomes
\begin{equation}\label{err1d}
  u-u^h = \sum_{n=1}^{N-1}(\hat{u}_n-\hat{u}_n^h)\sin(n\pi x)%
  =\sum_{n=1}^{N-1}\left(\hat{\mathcal{H}}_n^{-1}-(\hat{\mathcal{H}}_n^h)^{-1}\hat{\mathcal{R}}^h_n\right)\hat{f}_n\sin(n\pi x),
\end{equation}
where the Fourier transform of $\mathcal{R}^hf$ is supposed to be
$\widehat{(\mathcal{R}^hf)}_n=\hat{\mathcal{R}}^h_n\hat{f}_n$.

\begin{theorem}\label{1derr}
  Suppose $k\not\in\pi\mathbb{Z}$, $f\in H^p_0(0,1)$ and $f$ is band limited with the Fourier
  coefficients $\hat{f}_n=0$ for $n\ge N$ for an even integer $N\ge 4$. Let $u$ be the solution of
  \eqref{helm1d}. Suppose a linear finite difference equation $\mathcal{H}^hu^h=\mathcal{R}^hf$ on
  $\Omega_h$ with $u^h(0)=u^h(1)=0$ has a unique solution $u^h$, where
  $\Omega_h:=\{x_j=jh, j=1,\ldots,N-1\}$ and $h=\frac{1}{N}$. Let $u^h$ be extended by the discrete
  Fourier transform to $\Omega=(0,1)$. Then we have
  \[
    \|u-u^h\|_{L^2(0,1)}\le |f|_{H^p(0,1)}\max_{n=1,..,N-1}%
    \left|\hat{\mathcal{H}}_n^{-1}-(\hat{\mathcal{H}}_n^h)^{-1}\hat{\mathcal{R}}^h_n\right|%
    /|\xi_n|^p,
  \]
  \[
    |u-u^h|_{H^1(0,1)}\le |f|_{H^p(0,1)}\max_{n=1,..,N-1}%
    \left|\hat{\mathcal{H}}_n^{-1}-(\hat{\mathcal{H}}_n^h)^{-1}\hat{\mathcal{R}}^h_n \right|%
    /|\xi_n|^{p-1},
  \]
  where $|v|_{H^p(0,1)}=\|v^{(p)}\|_{L^2(0,1)}$ is the $H^p$-semi-norm, $\xi_n=n\pi$, and both upper
  bounds are attainable.
\end{theorem}

\begin{proof}
  The first inequality follows from the Parseval's identity in view of \eqref{err1d}
  \[
    2\|u-u^h\|_{L^2{(0,1)}}^2=\sum_{n=1}^{N-1}
    \left|\hat{\mathcal{H}}_n^{-1}-(\hat{\mathcal{H}}_n^h)^{-1}\hat{\mathcal{R}}^h_n\right|^2
    |\hat{f}_n|^2.
  \]
  The other inequalities can be proved in a similar way.
\end{proof}

\begin{remark}\label{f1d}
  For an $m$th order scheme, we usually choose $p=m-2$ in Theorem~\ref{1derr}. Sometimes, it is
  useful to choose $p=m-1$ for the $H^1$-semi-norm of the error. We shall see when $p\ge 1$ the
  $H^p$-semi-norm of $f$ does contribute extra powers of $k$ to the upper bound, because 1) the
  upper bound is attained when $\hat{f}_n$ is supported on the frequency where the maximum term in
  the upper bound is attained, and 2) the maximum term is often attained at a frequency near the
  wavenumber $k$. More precisely, if the upper bound is attained when $f=\sin(\xi x)$ with $\xi$ of
  order $k$ then $|f|_{H^p(0,1)}$ is of order $k^p$. Thus, we consider as hard as possible a
  Helmholtz problem for each wavenumber $k$.
\end{remark}

\subsection{FDMs in 1D and their symbols}
\label{fdm1d}

The baseline is of course the classical centered scheme
\[
  \mathcal{H}_{cls}^hu^h:=({k}^2-\frac{2}{h^2})u^h(x_j) + \frac{1}{h^2}u^h(x_{j-1})+ %
  \frac{1}{h^2}u^h(x_{j+1})=f(x_j),\quad x_j\in\Omega_h.
\]
Its left hand side symbol is $\hat{\mathcal{H}}^h_{cls,n}=k^2-\frac{4}{h^2}\sin^2\frac{\xi_nh}{2}$
with $\xi_n=n\pi$, and right hand side symbol is $\hat{\mathcal{R}}^h_{cls,n}=1$.  It is well known
that the dispersion error in 1D is avoidable; see \cite{babuska1997pollution}. Arbitrarily high
order 3-point schemes without dispersion in 1D was derived \cite{wang2014pollution}. We shall study
those schemes of order 2, 4, and 6. The 2nd order dispersion free scheme reads
\[
  \mathcal{H}_{df}^hu^h:={k}^2u^h(x_j) + \frac{k^2}{\tilde{k}^2}
  \left[-\frac{2}{h^2}u^h(x_j)+ \frac{1}{h^2}u^h(x_{j-1})+ %
    \frac{1}{h^2}u^h(x_{j+1})\right]=f(x_j),\quad x_j\in\Omega_h,
\]
where $\tilde{k}=\frac{2}{h}\sin\frac{kh}{2}$. Its left hand side symbol is
$\hat{\mathcal{H}}^h_{df,n}=k^2-\frac{k^2}{\tilde{k}^2}\frac{4}{h^2}\sin^2\frac{\xi_nh}{2}$. The
fourth order dispersion free scheme uses the same left hand side operator as the second order scheme
uses but different right hand side operator:
\[
  \mathcal{H}_{df}^hu^h=f(x_j)+\left(\frac{1}{\tilde{k}^2}-\frac{1}{k^2}\right)f''(x_j)%
  =:(\mathcal{R}_{df4}^hf)(x_j),\quad x_j\in\Omega_h.
\]
Its right hand side symbol is
$\hat{\mathcal{R}}^h_{df4,n}=1-\xi_n^2\left(\frac{1}{\tilde{k}^2}-\frac{1}{k^2}\right)$.  For the
sixth order dispersion free scheme, the left hand side is the same as before, and the right hand
side operator is
\[
  (\mathcal{R}_{df6}^hf)(x_j) = (\mathcal{R}_{df4}^hf)(x_j) + \left[\frac{1}{k^4} + %
  \left(\frac{h^2}{12}-\frac{1}{k^2}\right)\frac{1}{\tilde{k}^2}\right]f^{(4)}(x_j).
\]
So its symbol is $ \hat{\mathcal{R}}^h_{df6,n}=\hat{\mathcal{R}}^h_{df4,n} + %
\xi_n^4\left(\frac{h^2}{12}-\frac{1}{k^2}\right)\frac{1}{\tilde{k}^2}.  $

\subsection{Accuracy of FDMs in 1D}
\label{acc}

According to \eqref{err1d} and Theorem~\ref{1derr}, the error is essentially determined by
\[
  \psi_p(\xi_n):=\left|\hat{\mathcal{H}}_n^{-1}-(\hat{\mathcal{H}}_n^h)^{-1}
    \hat{\mathcal{R}}^h_n\right|/|\xi_n|^p,\quad \xi_n=n\pi, \quad n=1,\ldots,N-1.
\]
Recall Remark~\ref{f1d} that if $\max_n\psi_p(\xi_n)$ for the $L^2$-norm error, or
$\max_n\psi_{p-1}(\xi_n)$ for the $H^1$-semi-norm error, is attained near $k$ then $|f|_{H^p(0,1)}$
will be of order $k^p$. In that case, we shall multiply $\psi_p$ or $\psi_{p-1}$ with $k^p$ to
represent (up to a constant factor) the sharp upper bound of the $L^2$-norm or $H^1$-semi-norm
error. This gives rise to Fig.~\ref{1dfig} for the FDMs described in Sect.~\ref{fdm1d}, where
$p=m-2$ is used for an $m$th order (in $h$) scheme. In the first (second) column of the figure, we
show $\psi_p$ ($\psi_{p-1}$) for $L^2$- ($H^1$-semi-) norm error, from which, we can see clearly how
the error is distributed over the Fourier frequencies $\xi_n$. In particular, all $\max_n\psi_{p}$
are attained at a $\xi_n$ near $k$. In the third column, we present the scaling results of
$k^p\max\psi$.

The first row illustrates the convergence of classical scheme in term of $h$-refinement with fixed
$k$, which confirms the order $O(h^2)$ for both $L^2$- (indicated by $\max\psi_0$) and $H^1$-semi-
(indicated by $\max\psi_{-1}$) norms of the error. Since the order in $h$ is well-known, we shall
not illustrate that for the other schemes.

The second row is again for the classical scheme, but $k$ doubles while $N=h^{-1}$ quadruples. It is
seen that $\psi_{-1}$ converges slower than $\psi_{0}$ in this setting. In particular, the subfigure
in row 2 and column 3 shows that the order of $L^2$-norm error (indicated by $\max\psi_0$) is
$k^2h^2$ for the Fourier interpolated $u^h$ in \eqref{uh}, while the plots for $\max\psi_{-1}$ needs
some explanation. While $k$ doubles and $h^{-1}$ quadruples simultaneously, it follows that $kh$
halves, and $\max\psi_{-1}$ halves (linear in $kh$) seen from the subfigure. Since the order in $h$
is 2, the order in $k$ and $h$ should be $k^{\alpha}h^2$ for some $\alpha$. Since
$(2k)^{\alpha}(h/4)^2=k^{\alpha}h^22^{\alpha-4}$, we should have $\alpha-4=-1$ i.e. $\alpha=3$ and
the order of $H^1$-semi-norm error is $k^3h^2$. The conclusion drawn here can be proved rigorously
based on Theorem~\ref{1derr}, which will appear elsewhere.

The other rows are for the dispersion free schemes \cite{wang2014pollution} described in
Sect~\ref{fdm1d}. It is seen that the $L^2$-norm error converges faster than the $H^1$-semi-norm
error in the setting that $k$ doubles while $h^{-1}$ quadruples. Since the order in $h$ is known,
the scaling with `wrong' order in $h$ in column 3 subfigures need some explanation. For the 2nd
order scheme, the $k^3h^3$ scaling means $k^2\max\psi_0$ scales down by one eighth when $k$ doubles
and $h$ quadruples simultaneously. So $(2k)^{\alpha}(h/4)^2=k^{\alpha}h^22^{\alpha-4}$ should have
$\alpha-4=-3$ i.e.  $\alpha=1$ and the $L^2$-norm error is of order $kh^2$. Similarly, the 4th (6th)
order scheme has the $L^2$-norm error of order $k^{3}h^4$ ($k^5h^6$) from $\alpha-8=-5$
($\alpha-12=-7$).

\begin{figure}
  \centering
  \includegraphics[height=12em,width=15em,trim=42 180 70 180,clip]{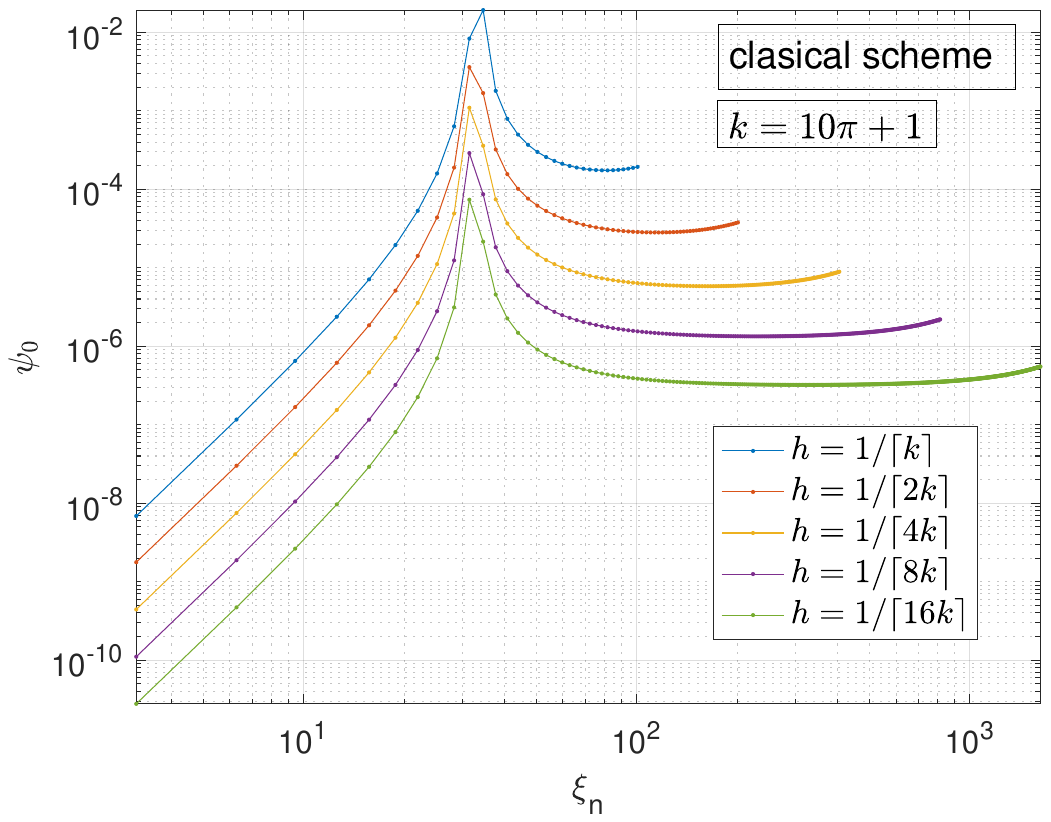} 
  \includegraphics[height=12em,width=15em,trim=42 180 70 180,clip]{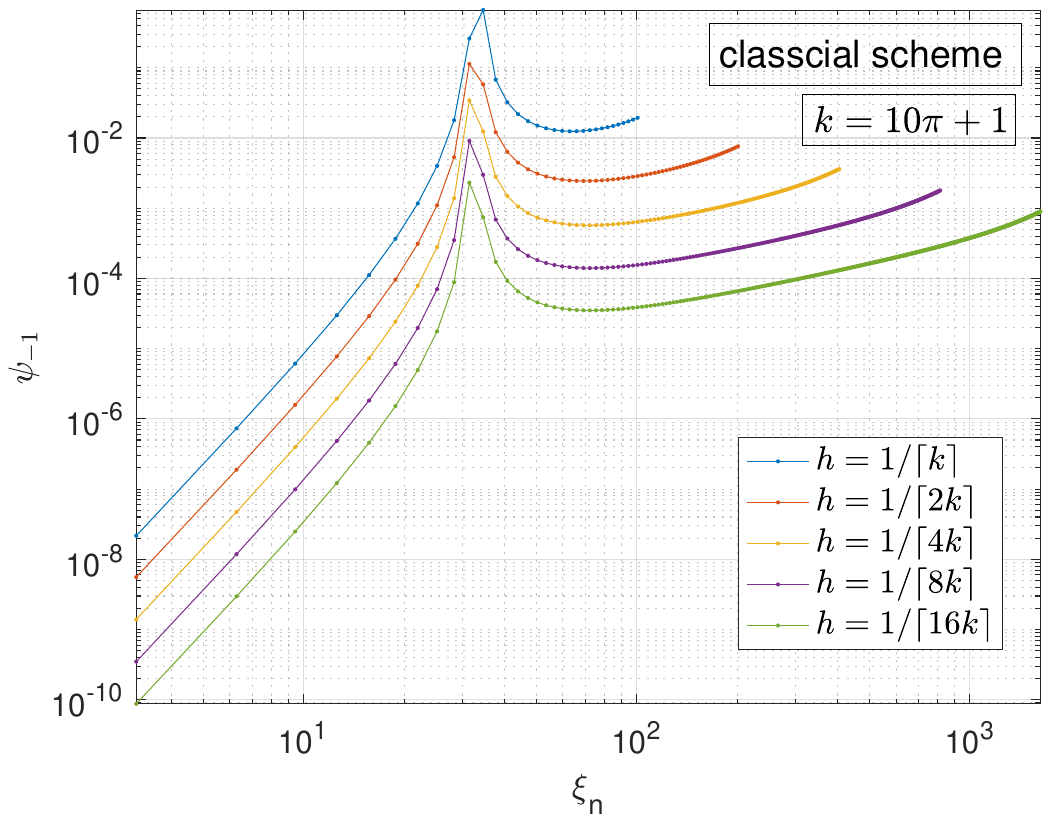}
  \includegraphics[height=12em,width=15em,trim=45 180 65 180,clip]{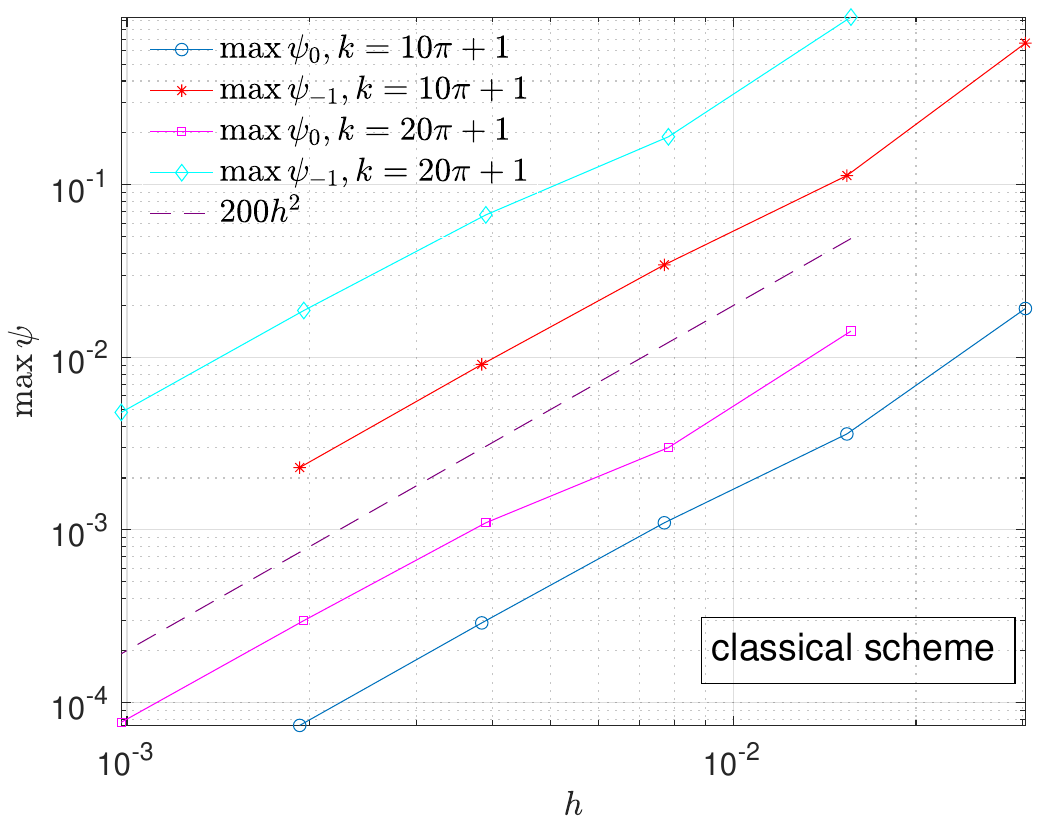}
  \includegraphics[height=12em,width=15em,trim=45 180 70 180,clip]{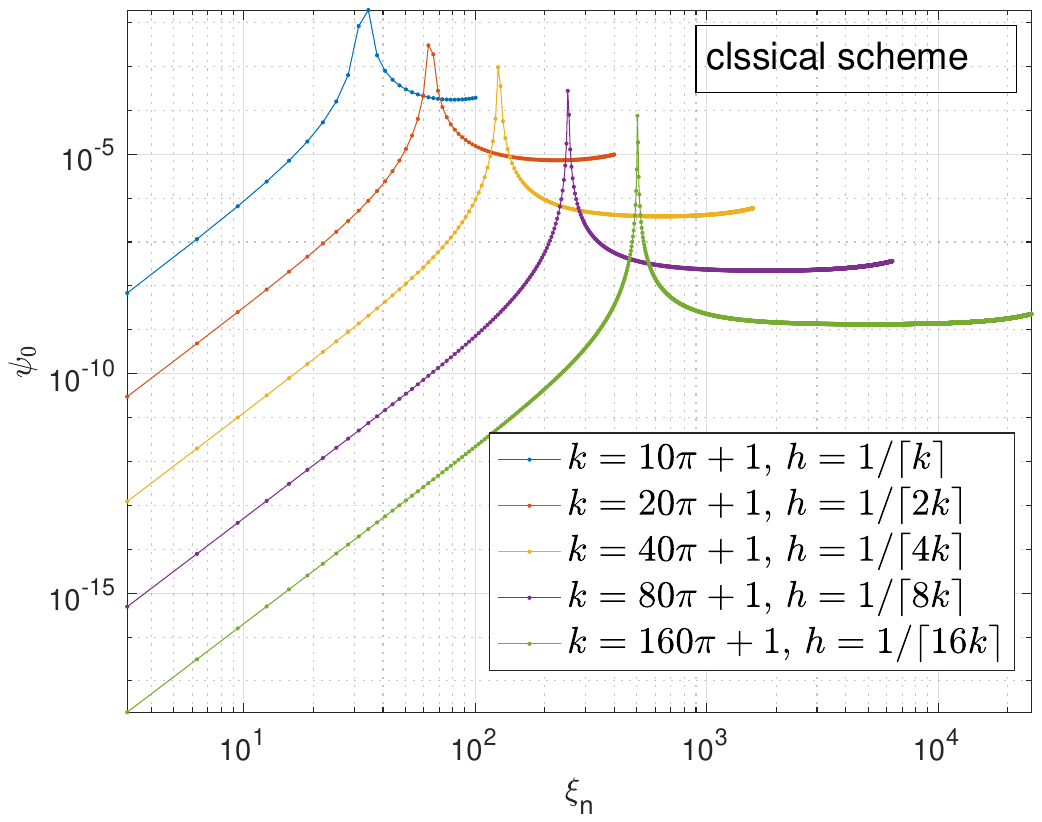} 
  \includegraphics[height=12em,width=15em,trim=45 180 70 180,clip]{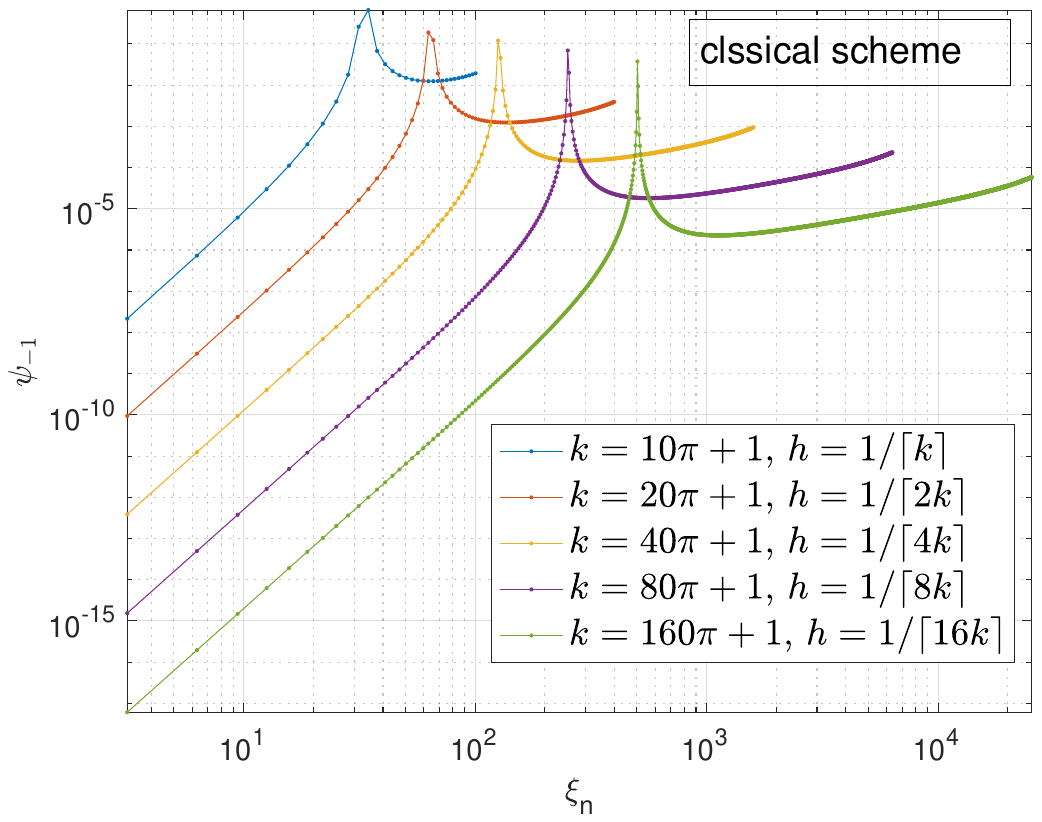}
  \includegraphics[height=12em,width=15em,trim=45 180 65 180,clip]{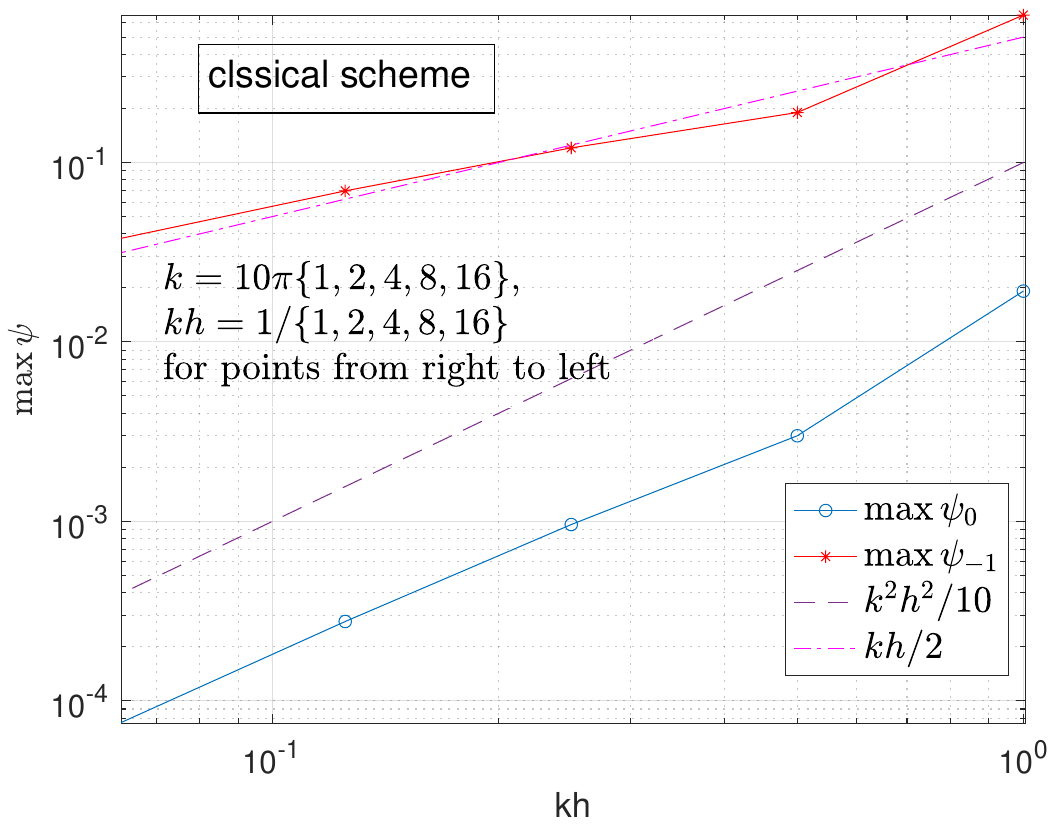}
  \includegraphics[height=12em,width=15em,trim=40 180 70 180,clip]{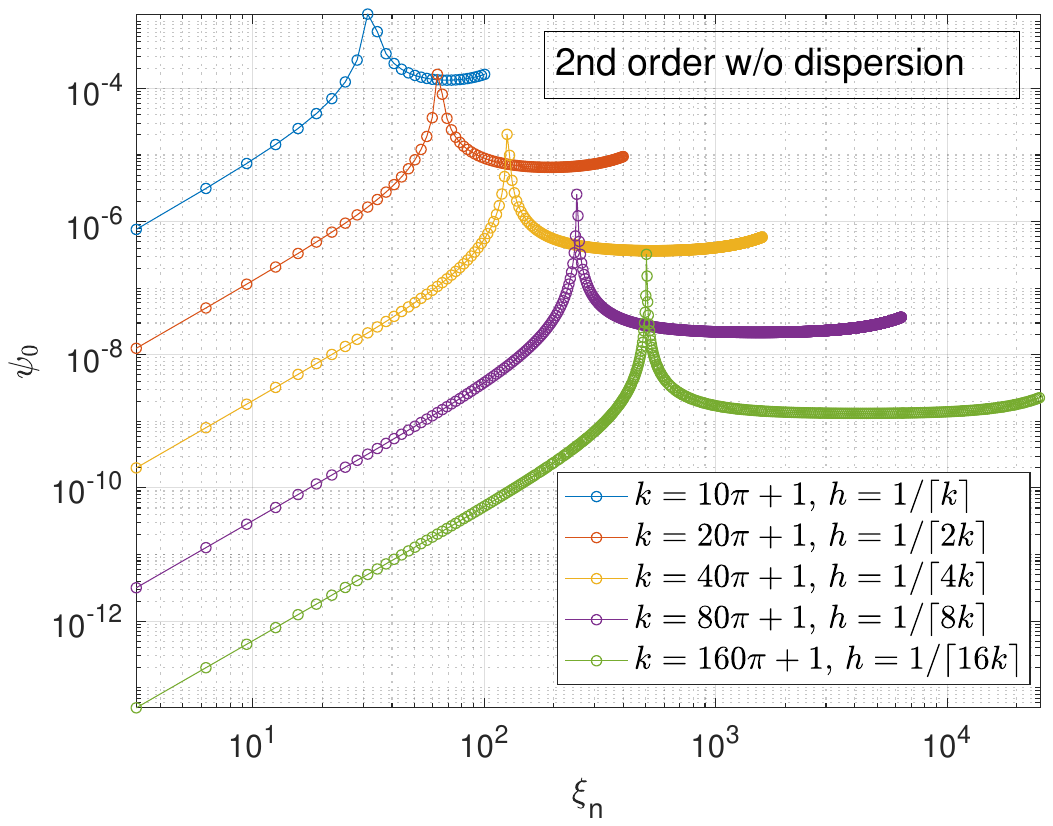}
  \includegraphics[height=12em,width=15em,trim=40 180 70 180,clip]{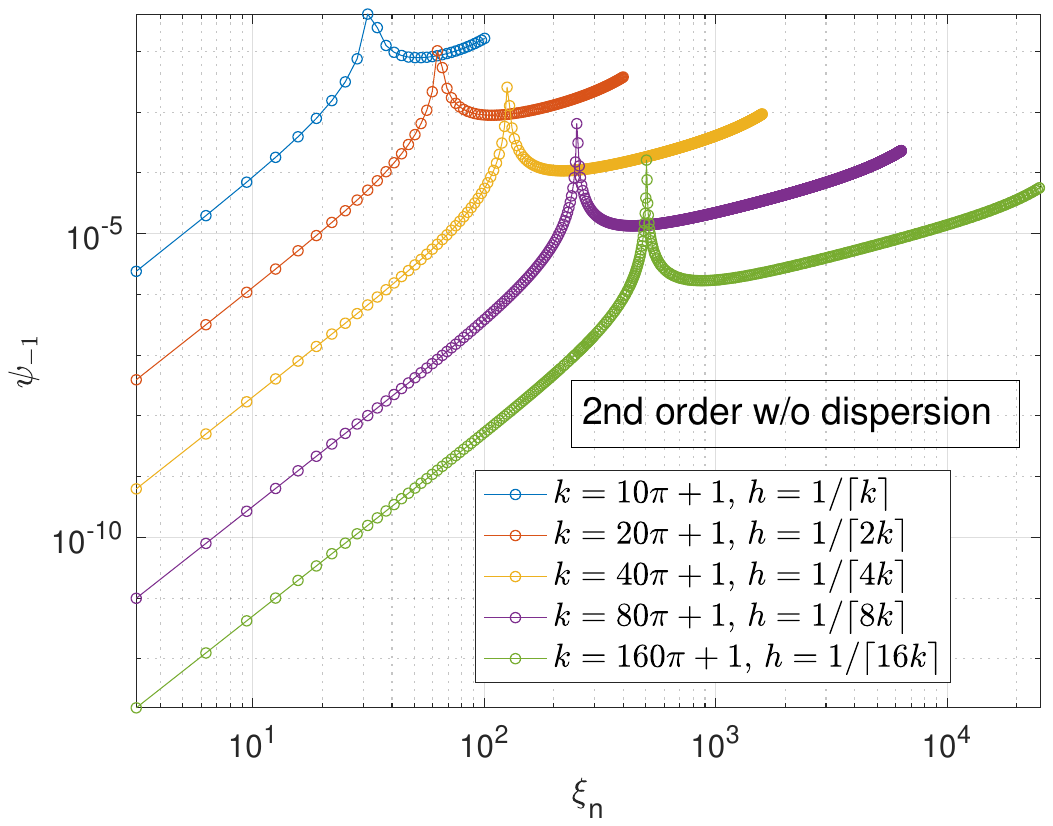}
  \includegraphics[height=12em,width=15em,trim=40 180 65 180,clip]{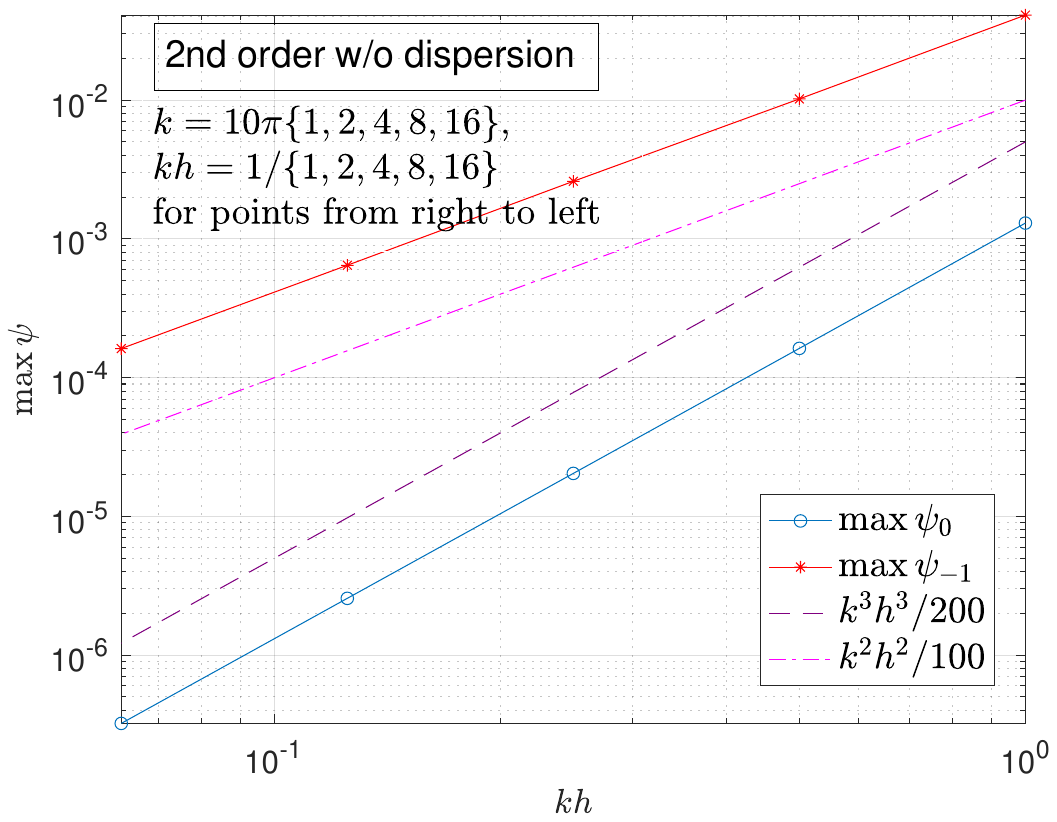}
  \includegraphics[height=12em,width=15em,trim=35 180 70 180,clip]{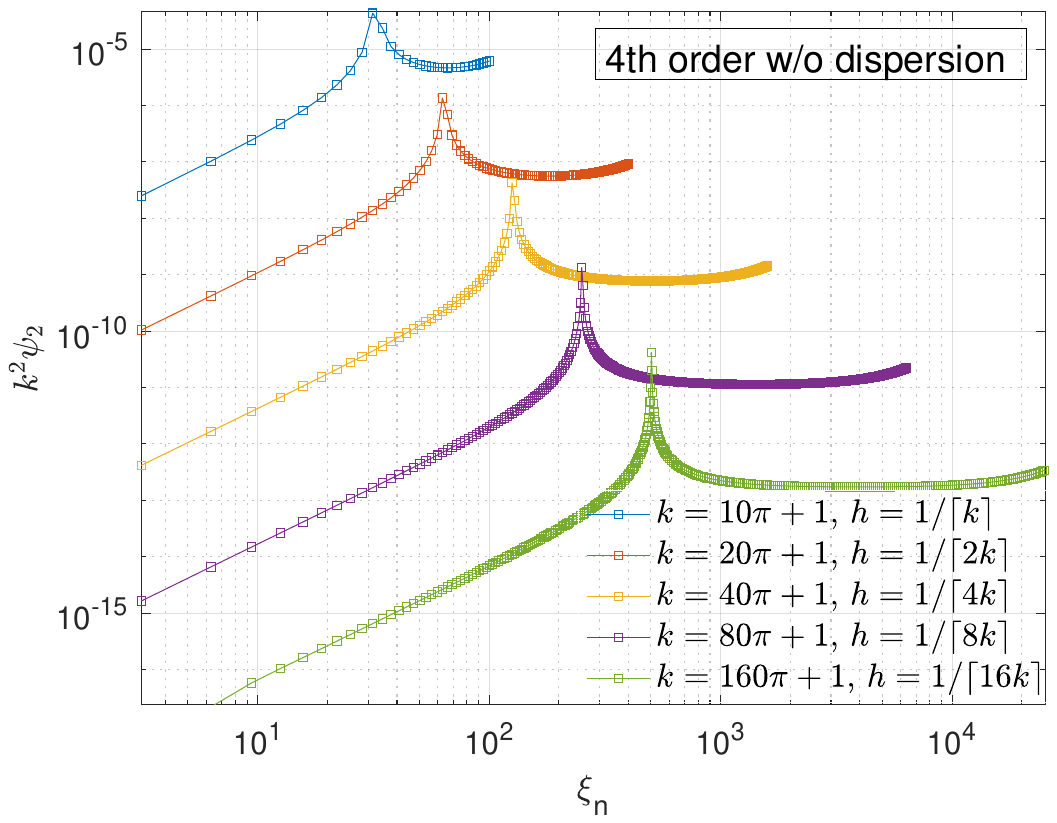}
  \includegraphics[height=12em,width=15em,trim=35 180 70 180,clip]{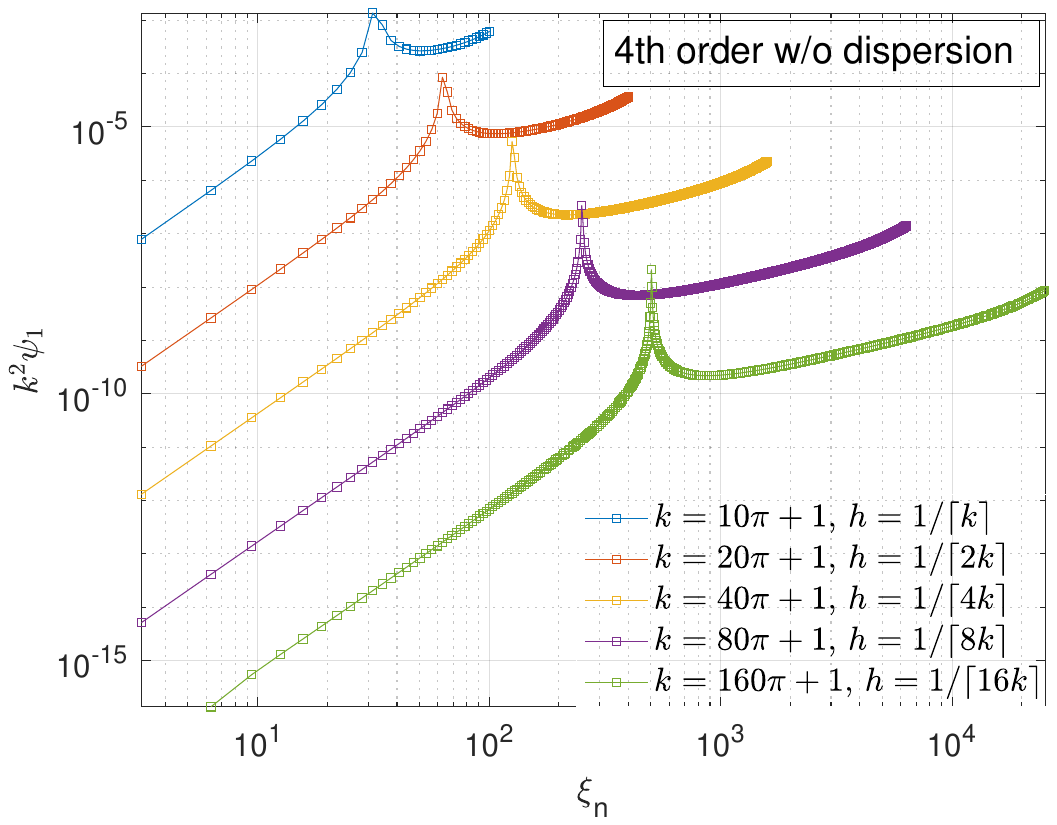}
  \includegraphics[height=12em,width=15em,trim=35 180 65 180,clip]{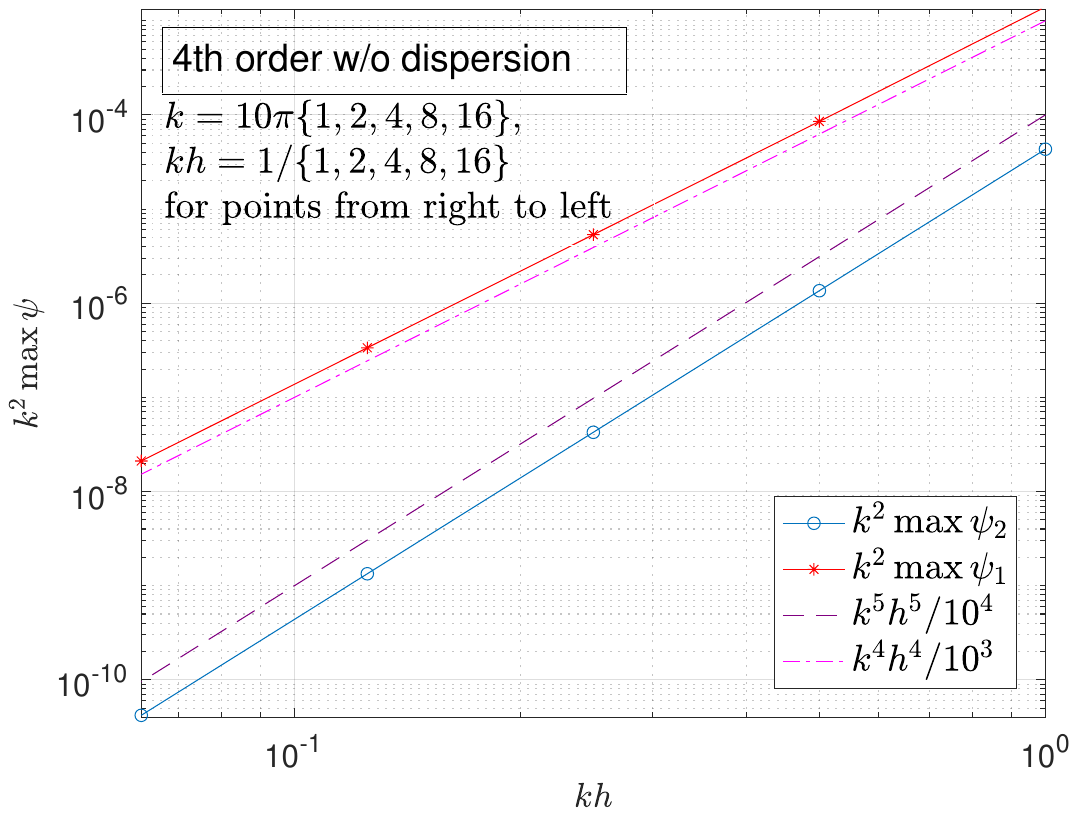}
  \includegraphics[height=12em,width=15em,trim=35 180 70 180,clip]{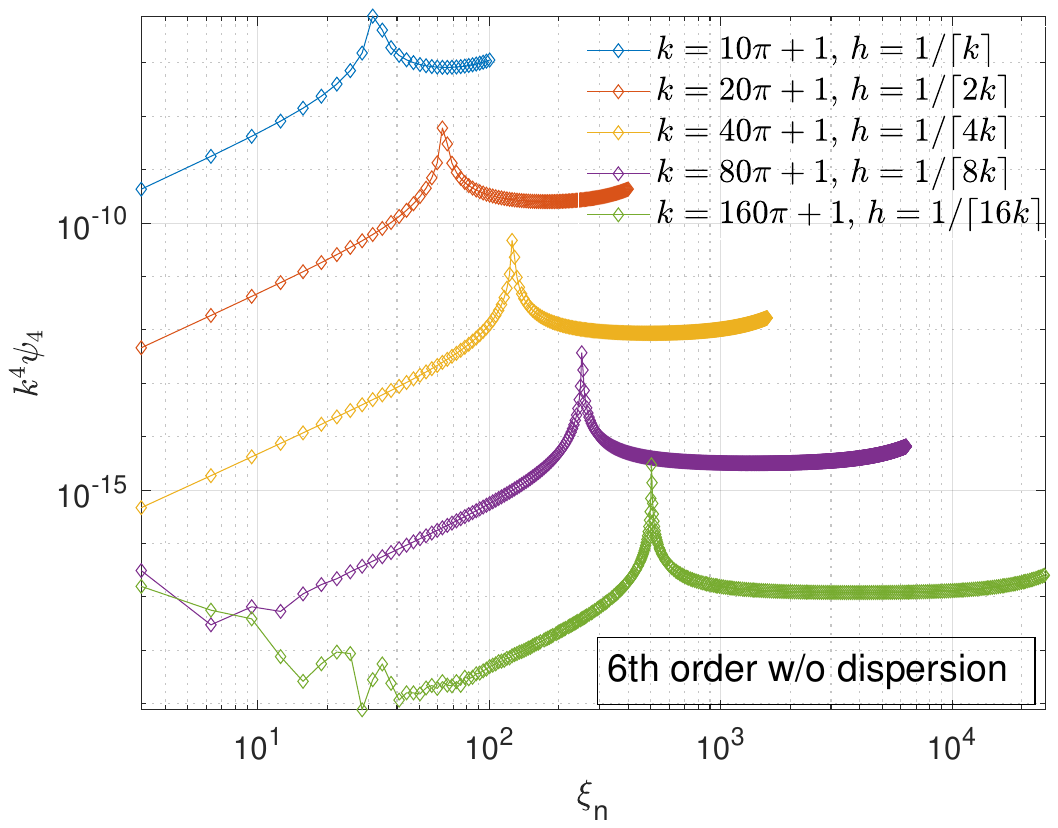}
  \includegraphics[height=12em,width=15em,trim=35 180 70 180,clip]{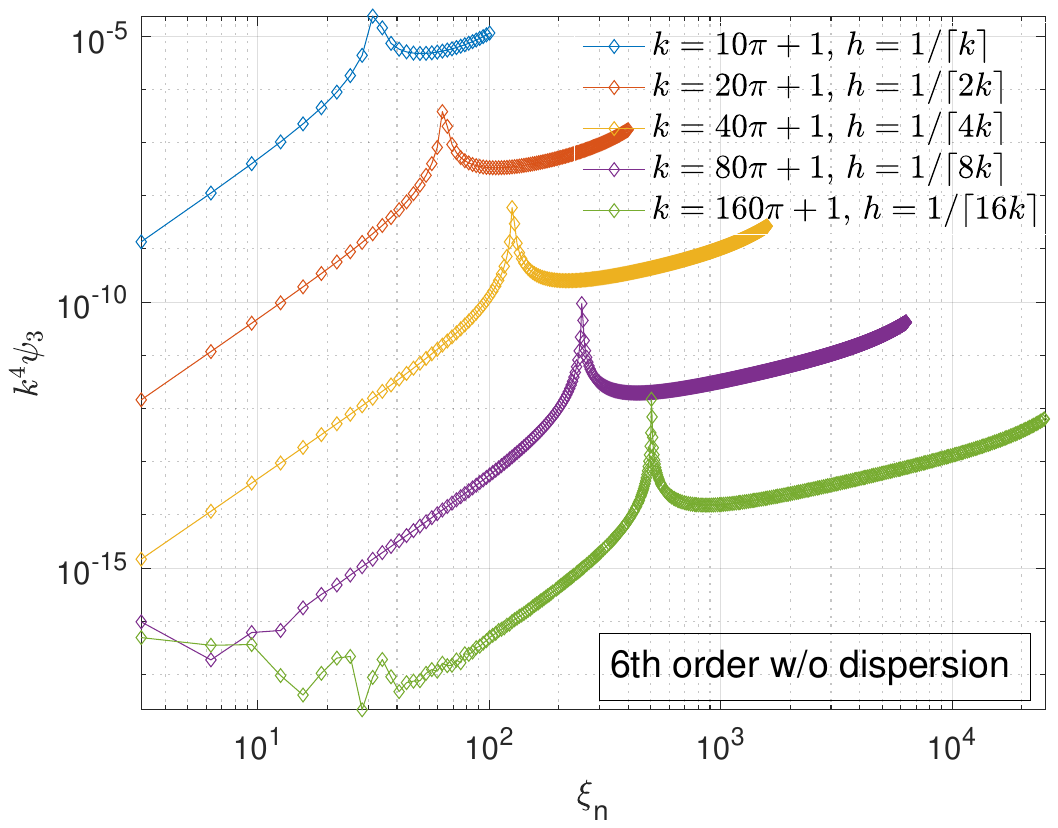}
  \includegraphics[height=12em,width=15em,trim=35 180 65 180,clip]{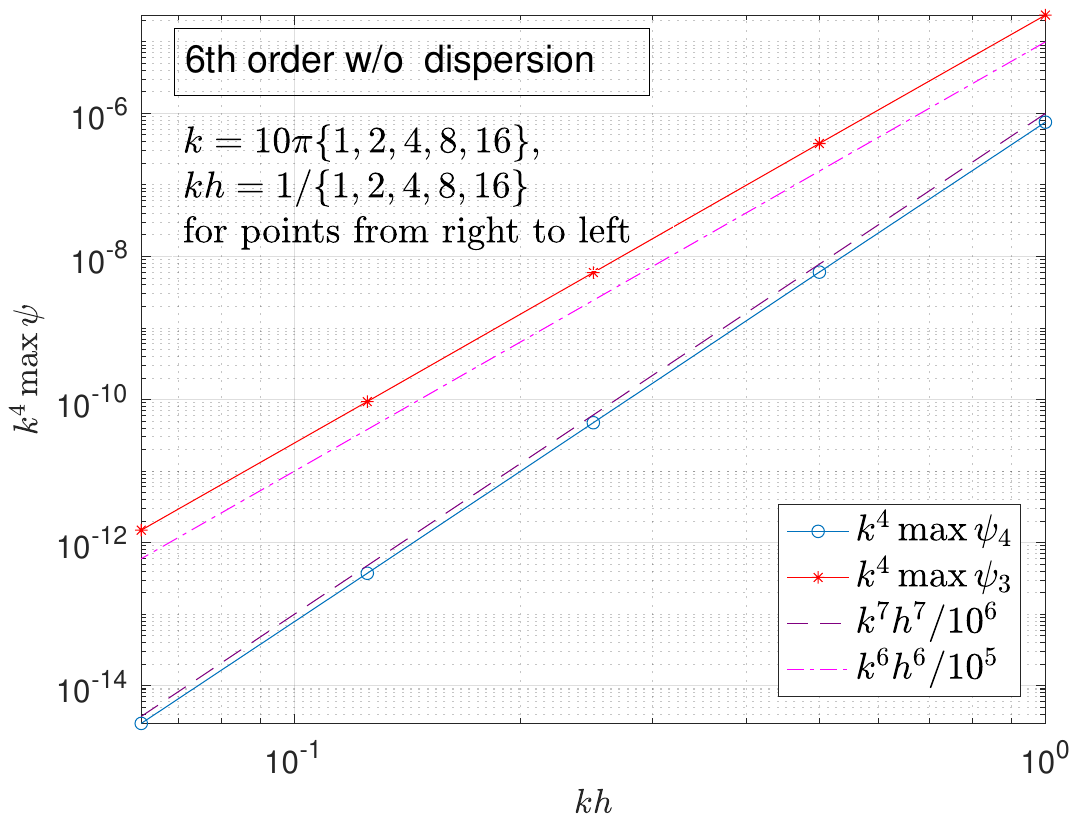}
  \caption{Symbol errors in 1D: $k^p\psi_p$ for $L^2$-norm (col.1) and $k^p\psi_{p-1}$ for
    $H^1$-semi-norm (col.2).}\label{1dfig}
\end{figure}


\section{Numerical experiments in 1D}

We solve \eqref{helm1d} numerically by the FDMs described in Sect.~\ref{fdm1d}.  Two different type
of sources are used. With $k=n_k\pi+1$ and a monochromatic source $f(x)=\sin(n_k\pi x)$, the exact
solution is $u(x)=f(x)/(k^2-n_k^2\pi^2)$. With a mixed source
$f(x)=\sum_{j=0}^5\sin(2^j\times 5\pi x)$, the exact solution is
$ u(x) = \sum_{j=0}^5\frac{1}{k^2-2^{2j}\times 25\pi^2}\sin(2^j\times 5\pi x)$. A range of $k$ and
$h$ values are tested. Denote the error $e^h:=u-u^h$. The error (semi-)norms $\|e^h\|_{L^2(0,1)}$
and $|e^h|_{H^1(0,1)}$ are computed in the Fourier frequency domain by Parseval's identity. The
results are shown in Fig.~\ref{1dnum}.

In the first row of Fig.~\ref{1dnum}, the error of classical scheme is presented. For the source
$f(x)=\sin(k-1)x$ of a single frequency near the wavenumber $k=n_k\pi+1$, the $L^2$-norm error in
the first column is seen of order $k^2h^2$ with $\|e^h\|_{L^2}/(k^2h^2)$ almost independent of $k$,
which corroborates Theorem~\ref{1derr}. It is only when $kh>0.2$ that larger $k$ leads to a smaller
$\|e^h\|_{L^2}$ at the same $kh$ value. This preasymptotic behavior can be understood from the
symbol error
$ \psi_0(k-1)=\hat{\mathcal{H}}^{-1}-(\hat{\mathcal{H}}_{cls,n_k})^{-1}=
\frac{1}{k^2-(k-1)^2}-\frac{1}{k^2-\frac{4}{h^2}\sin^2\frac{(k-1)h}{2}} $ which is about 0.0106384,
0.00832166, 0.00558938, 0.00329617, 0.00180466 for $k=5\pi+1, 10\pi+1, 20\pi+1, 40\pi+1, 80\pi+1$ at
$h=1/\lceil k\rceil$. The $H^1$-semi-norm error for the monochromatic source is shown in the second
column. Along each line for a fixed $k$, the subfigure shows that $|e^h|_{H^1}$ decays at the rate
of order $h^2$. But at the same $kh$ value, it shows that $|e^h|_{H^1}$ doubles when $k$ doubles,
thus confirms the finding in Sect.~\ref{acc} that $|e^h|_{H^1}$ is of order $k^3h^2$. The last two
columns of the first row for the mixed and fixed source are similar to the first two columns for the
monochromatic source, showing that the error is still concentrated at nearly resonant (i.e. near
$k$) Fourier frequencies.

In the second and subsequent rows of Fig.~\ref{1dnum}, the dispersion free schemes
\cite{wang2014pollution} described in Sect.~\ref{fdm1d} are studied. It is seen that, for both
monochromatic and mixed sources, the $L^2$-norm error is of the expected order $m$ in $h$ and order
$m-1$ in $k$, while the $H^1$-semi-norm error is almost a constant times $k^mh^m$, all agreeing with
Sect.~\ref{acc}. Only the last column needs some explanation. From $k=5\pi+1$ to $k=10\pi+1$ (and to
$k=20\pi+1$, etc.), $|e^h|_{H^1}$ experiences a drop at the same $kh$ value. For sufficiently small
$h$, the mixed source is actually well resolved. To understand the error drops as $k$ increases, we
go back to Sect.~\ref{acc} and take a closer look at row 3-5 and column 3 of Fig.~\ref{1dfig}. It
can be seen that the evanescent modes (with Fourier frequency $\xi_n$ greater than the wavenumber
$k$) have a larger error than the propagating modes (with $\xi_n<k$). Now let us plot the symbol
errors when $kh$ is fixed but $k$ increases; see Fig.~\ref{1dkhfixed}. We see that small $k$ has a
larger evanescent error. Those evanescent errors which dominate $|e^h|_{H^1}$ simply decreases while
$k$ increases and $kh$ is fixed.  In Fig.~\ref{1dnum}, it is also seen that, for small values of
$kh$, the round-off errors become apparent in both the 4th and 6th order schemes, indicating a
smallest $h$ usable in practice.

\begin{figure}
  \centering%
  \includegraphics[scale=.3,trim=30 180 75 180,clip]{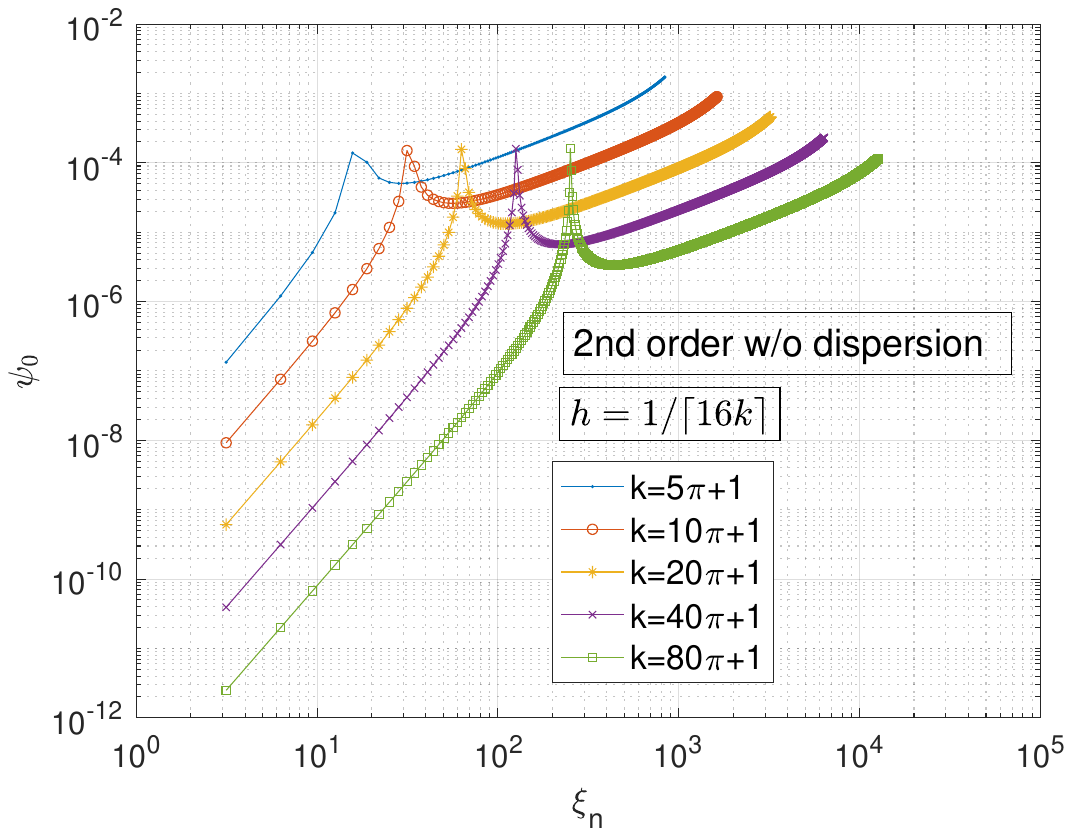}
  \includegraphics[scale=.3,trim=30 180 75 180,clip]{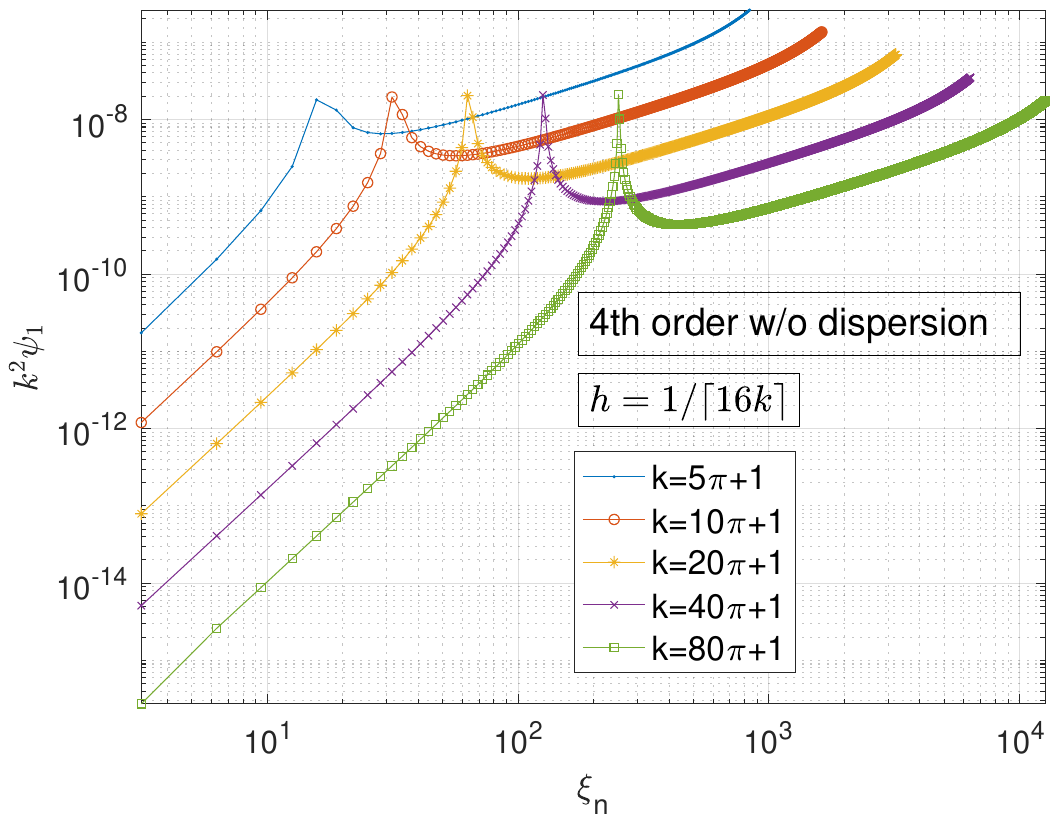}
  \includegraphics[scale=.3,trim=30 180 75 180,clip]{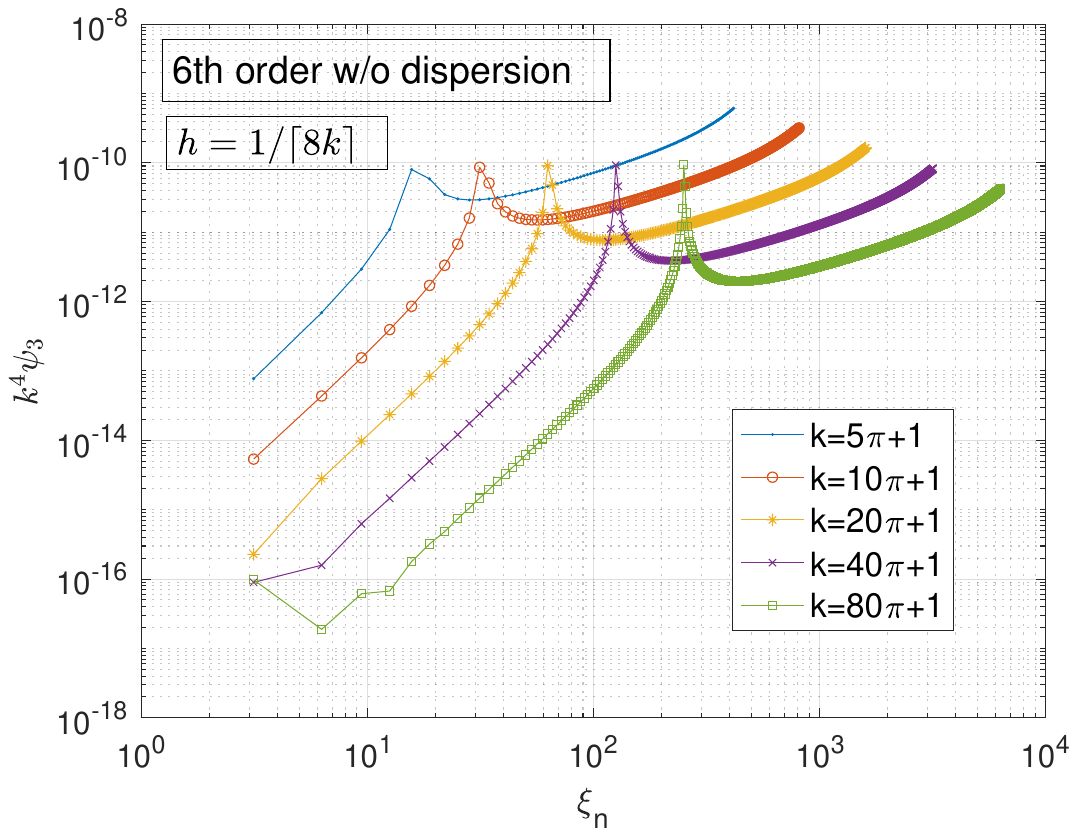}
  \caption{Symbol errors in 1D with $kh$ fixed and $k$ increases.}\label{1dkhfixed}
\end{figure}

\begin{figure}
  \centering %
  \includegraphics[height=14em, width=12em,trim=30 180 75 180,clip]{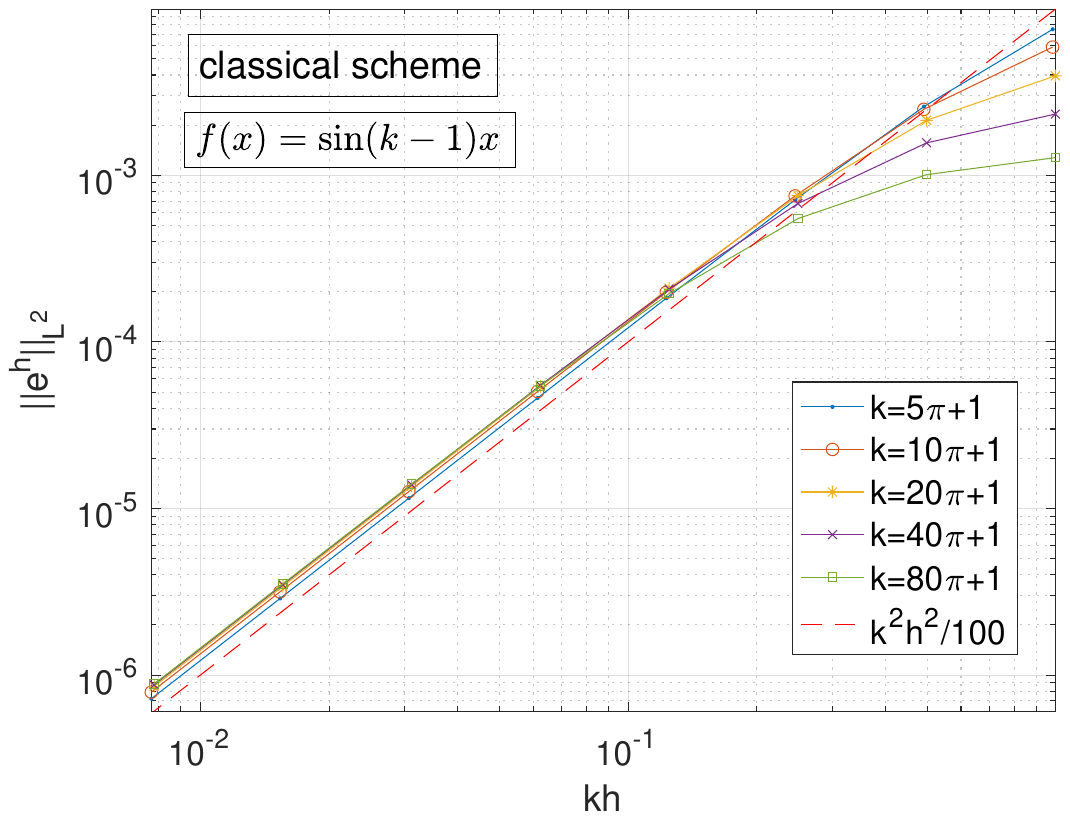}\;%
  \includegraphics[height=14em, width=12em,trim=30 180 75 180,clip]{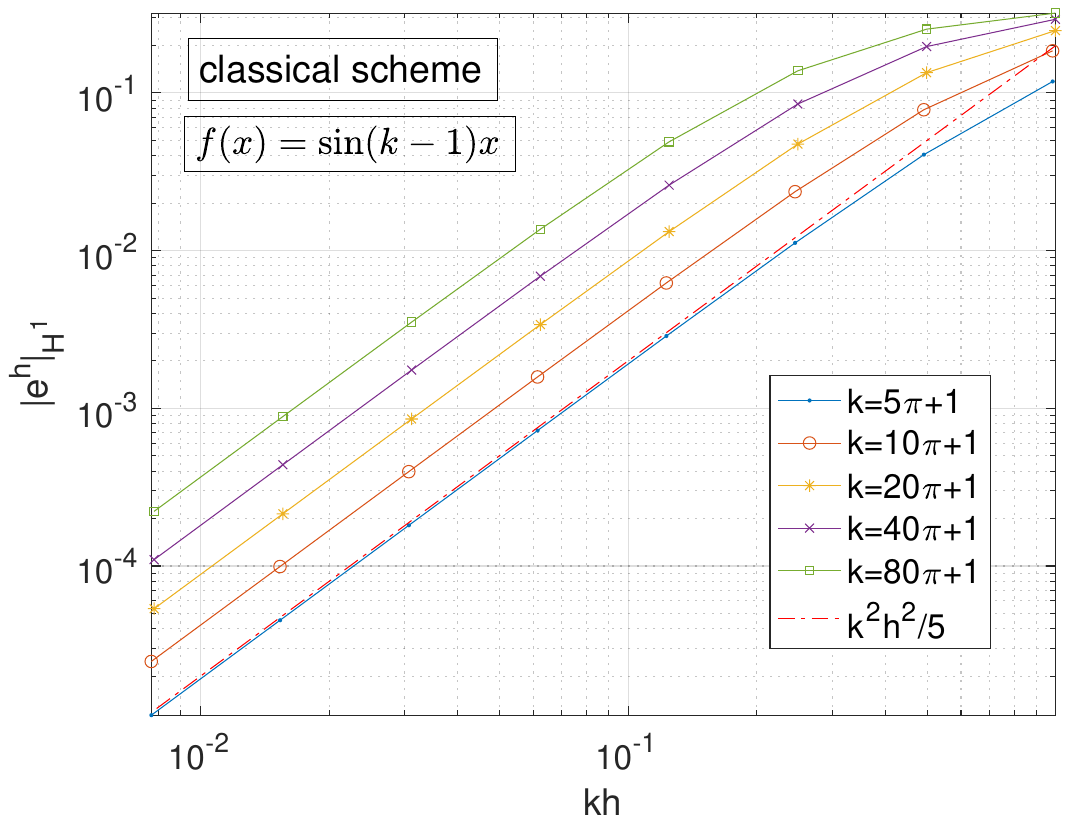}\;%
  \includegraphics[height=14em, width=12em,trim=30 180 75 180,clip]{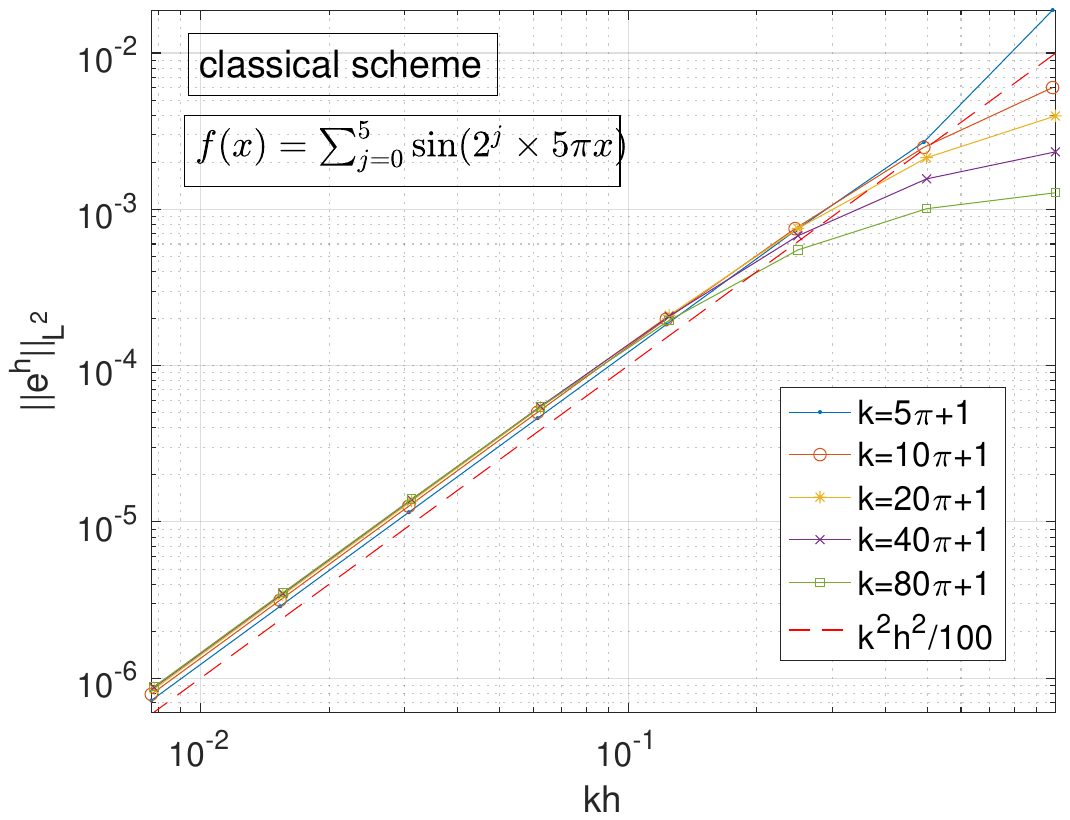}\;%
  \includegraphics[height=14em, width=12em,trim=30 180 75 180,clip]{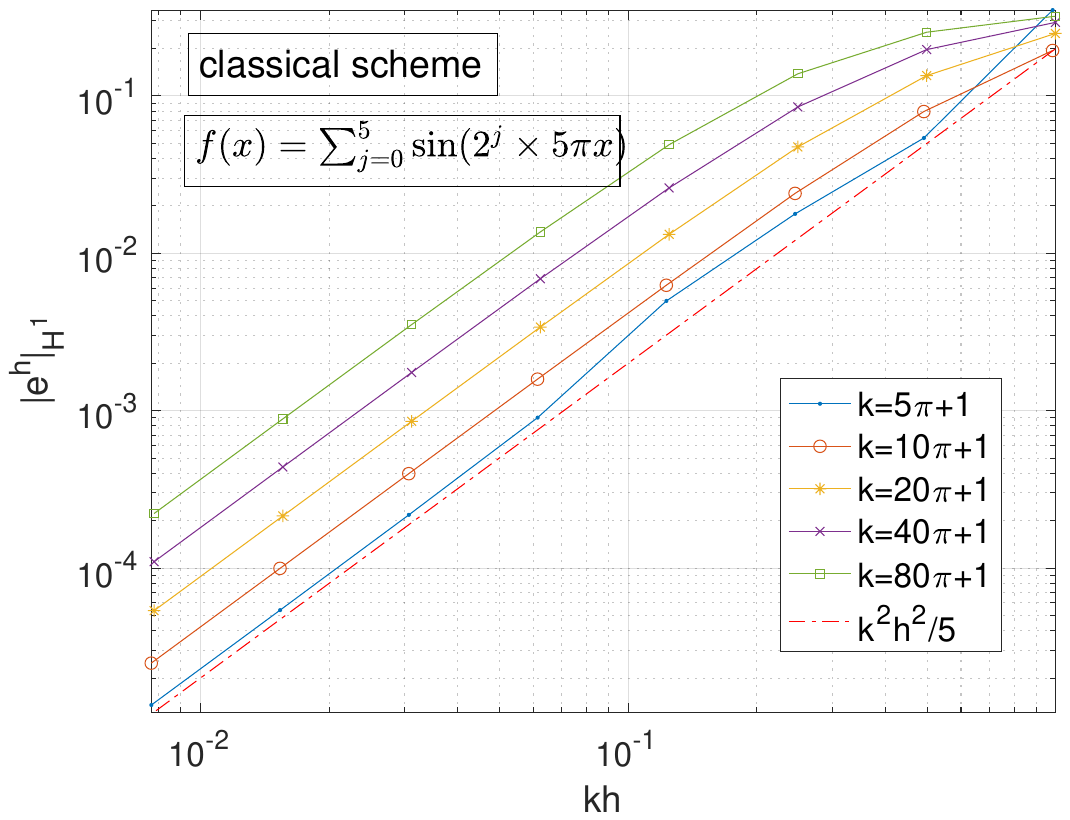}\\%
  \includegraphics[height=14em, width=12em,trim=30 180 75 180,clip]{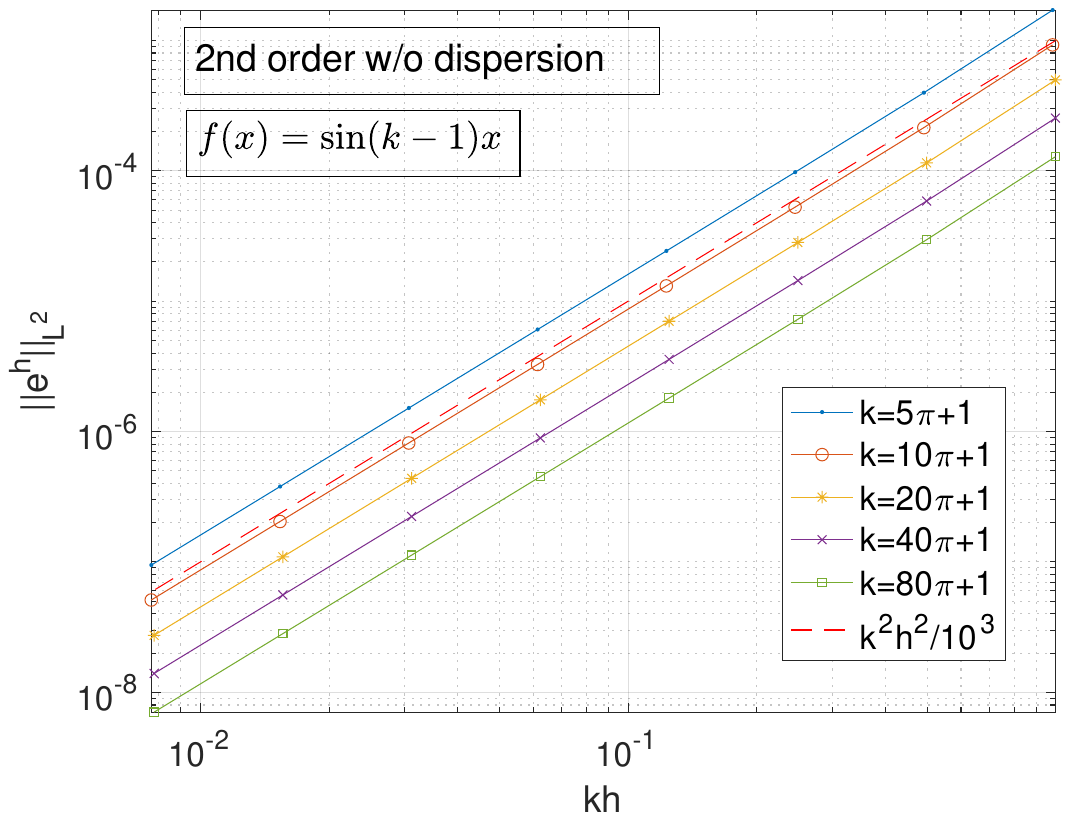}\;%
  \includegraphics[height=14em, width=12em,trim=30 180 75 180,clip]{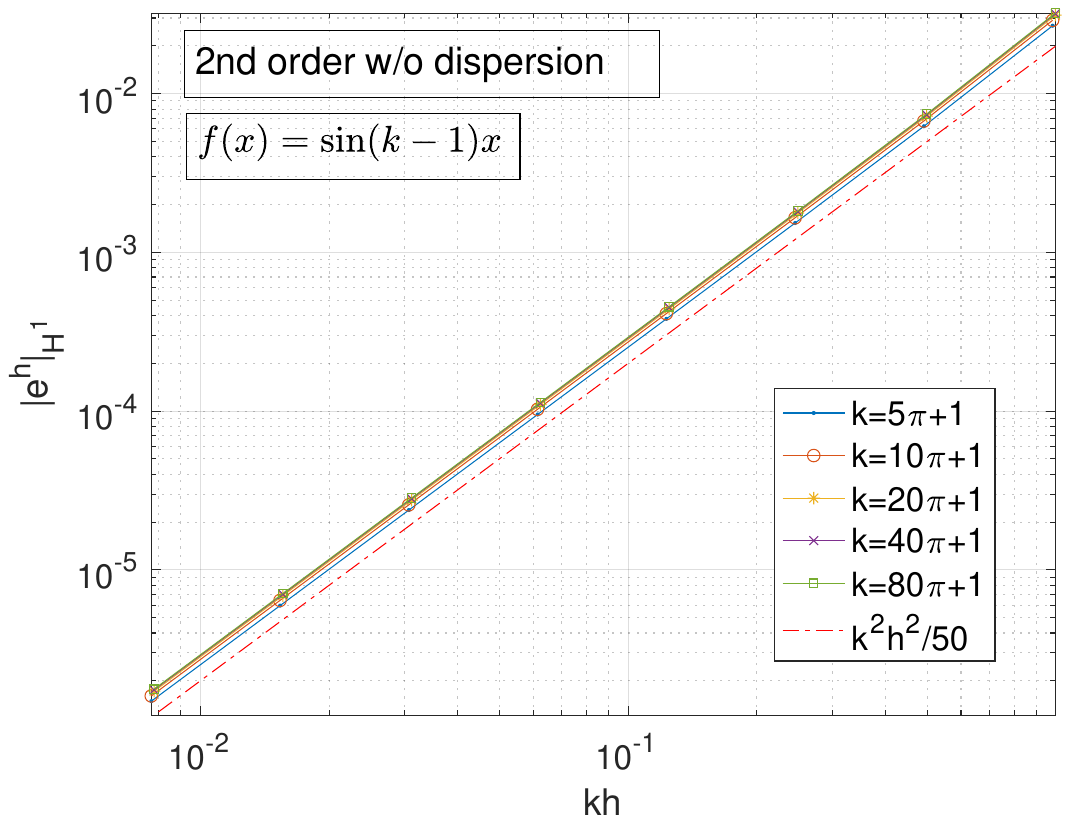}\;%
  \includegraphics[height=14em, width=12em,trim=30 180 75 180,clip]{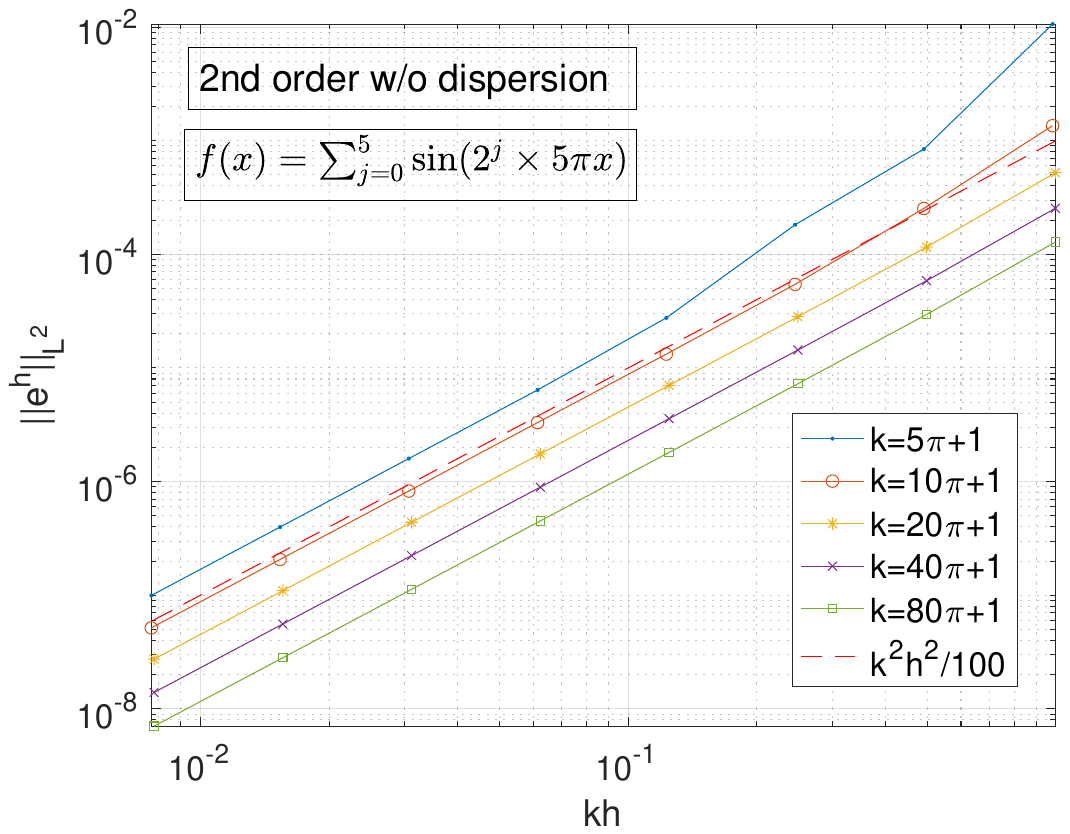}\;%
  \includegraphics[height=14em, width=12em,trim=30 180 75 180,clip]{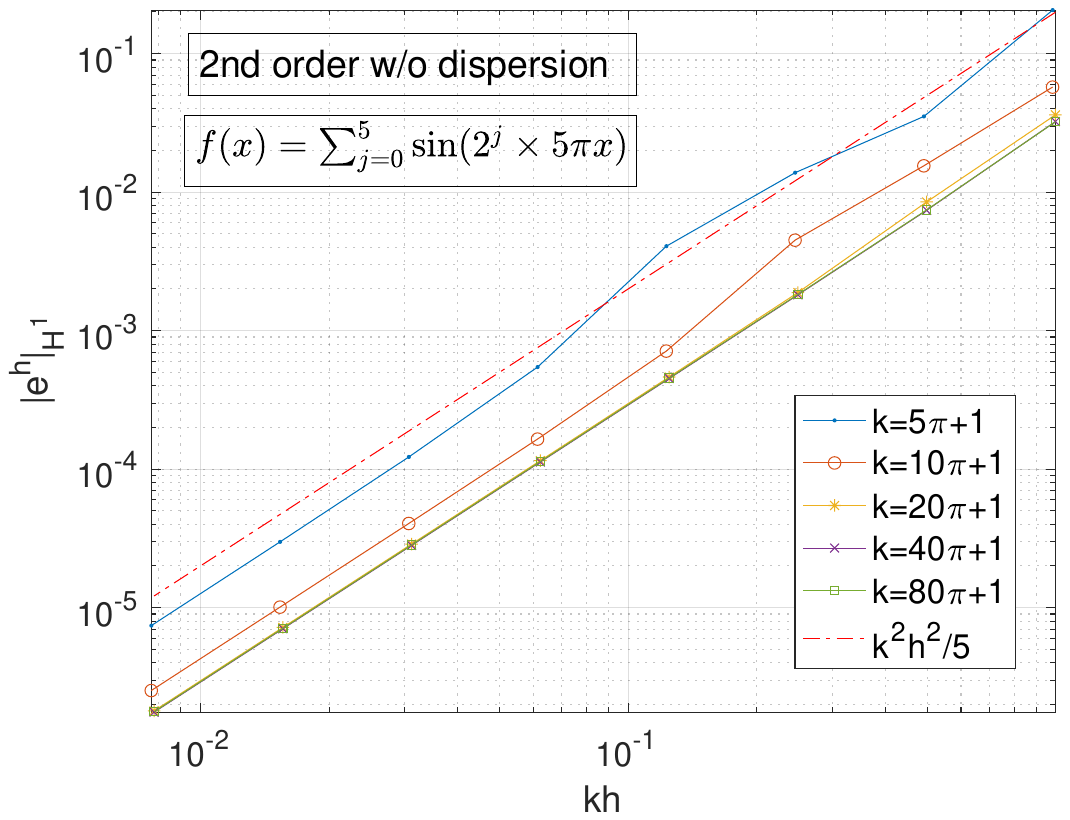}\\%
  \includegraphics[height=14em, width=12em,trim=30 180 75 180,clip]{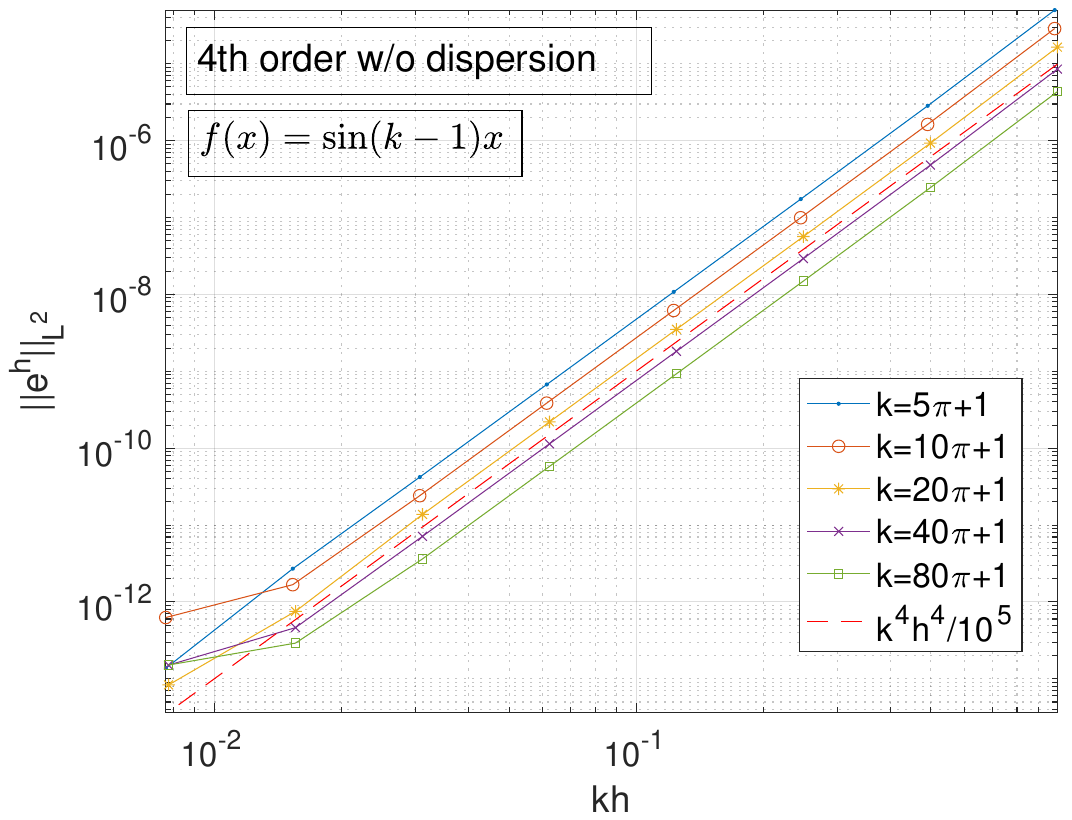}\;%
  \includegraphics[height=14em, width=12em,trim=30 180 75 180,clip]{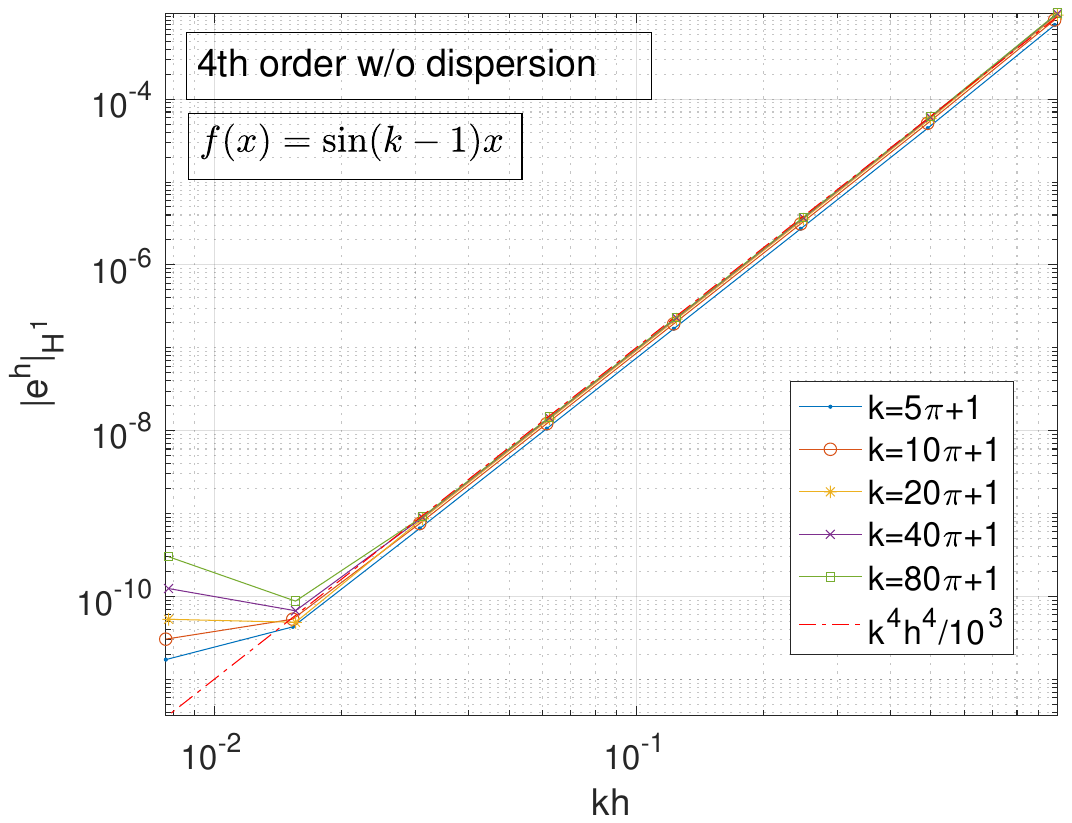}\;%
  \includegraphics[height=14em, width=12em,trim=30 180 75 180,clip]{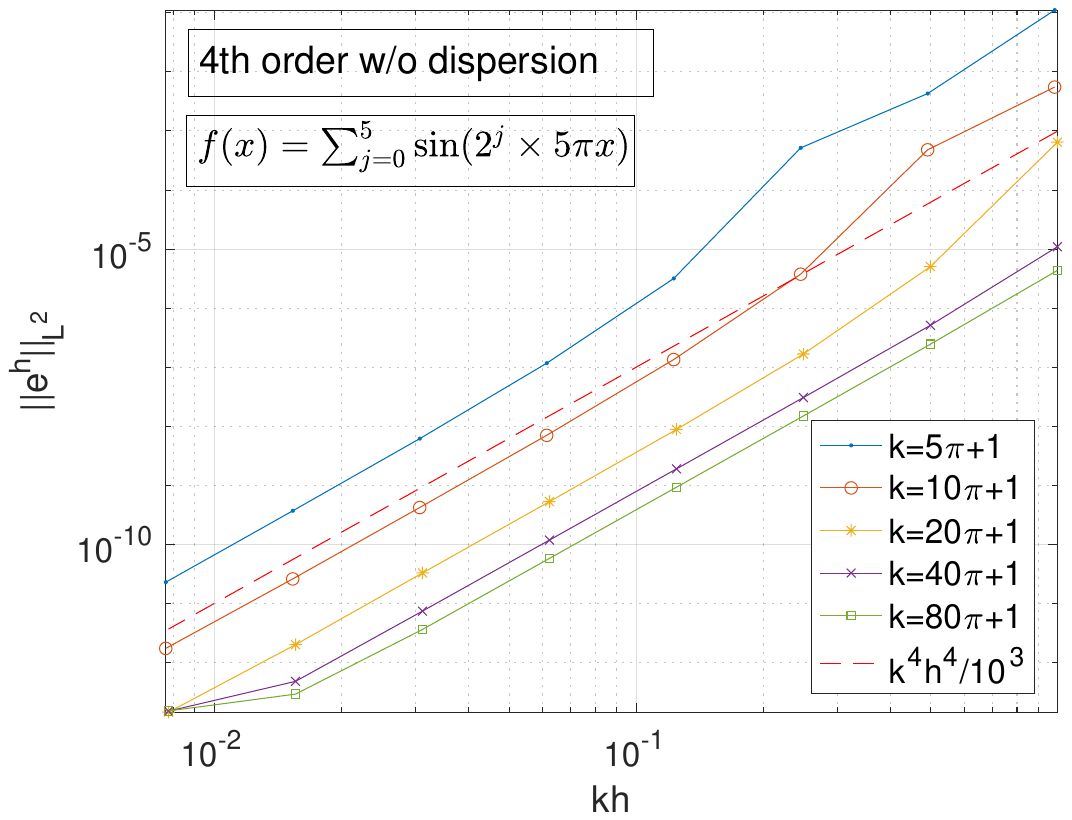}\;%
  \includegraphics[height=14em, width=12em,trim=30 180 75 180,clip]{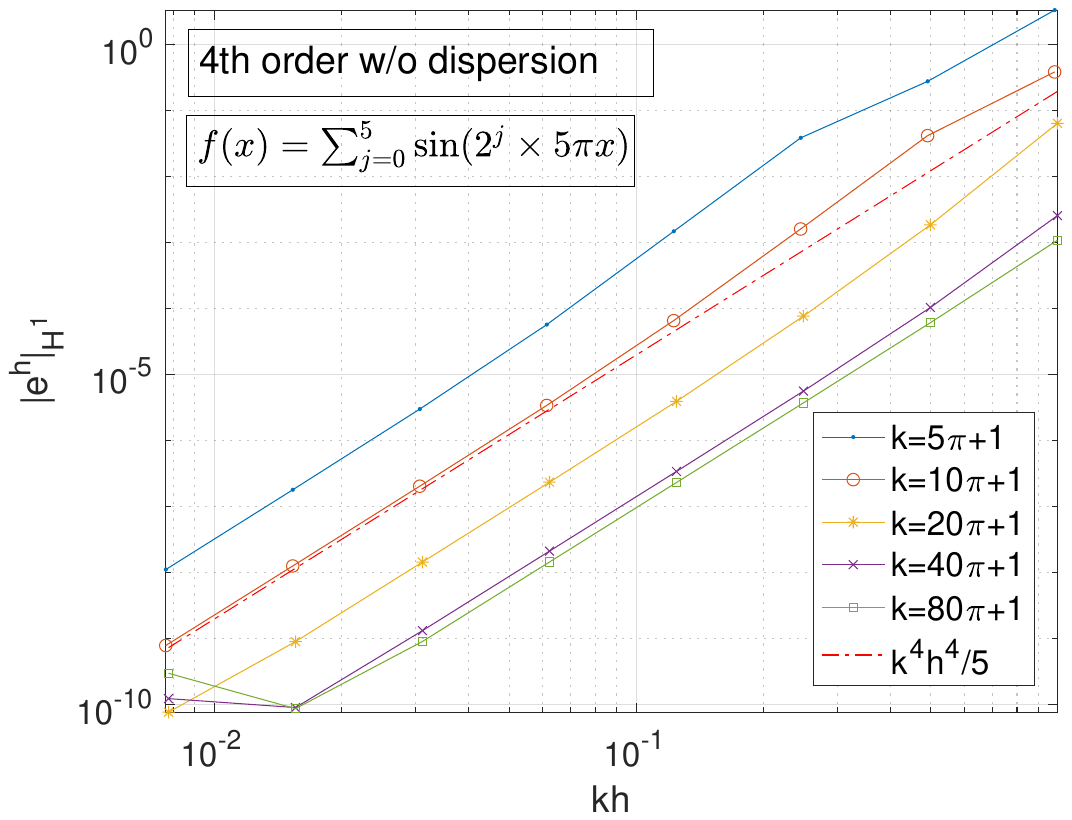}\\%
  \includegraphics[height=14em, width=12em,trim=30 180 75 180,clip]{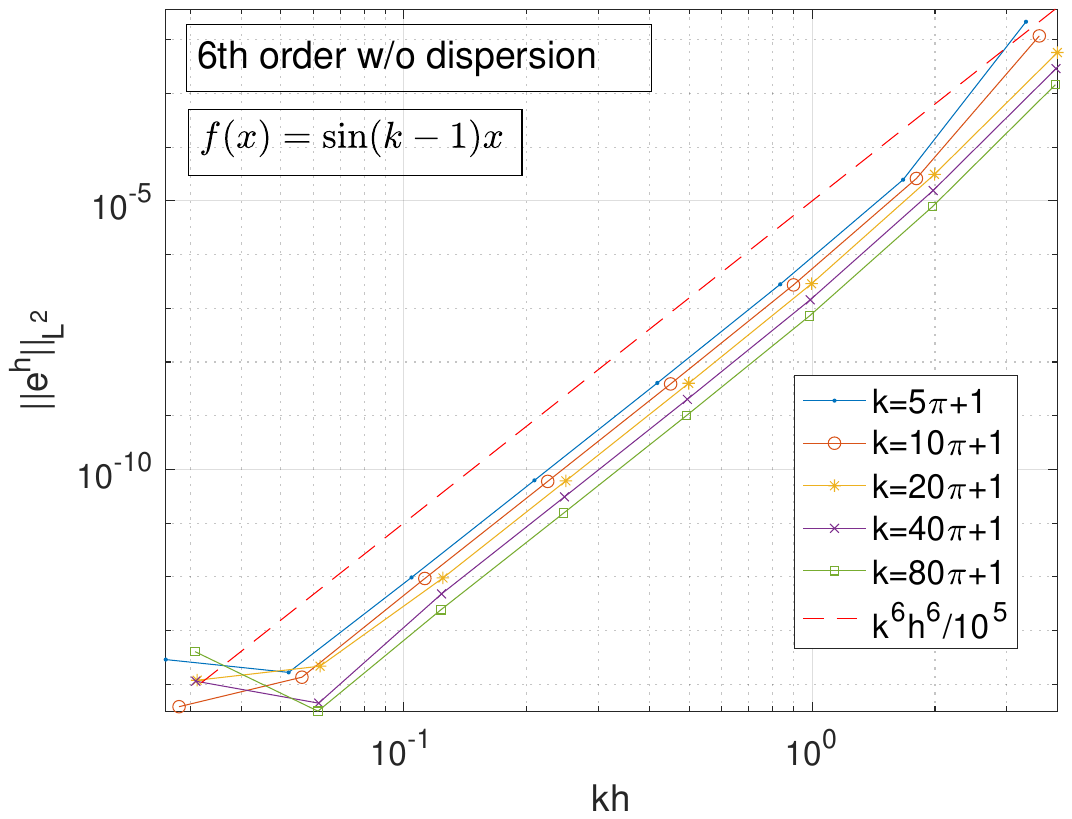}\;%
  \includegraphics[height=14em, width=12em,trim=30 180 75 180,clip]{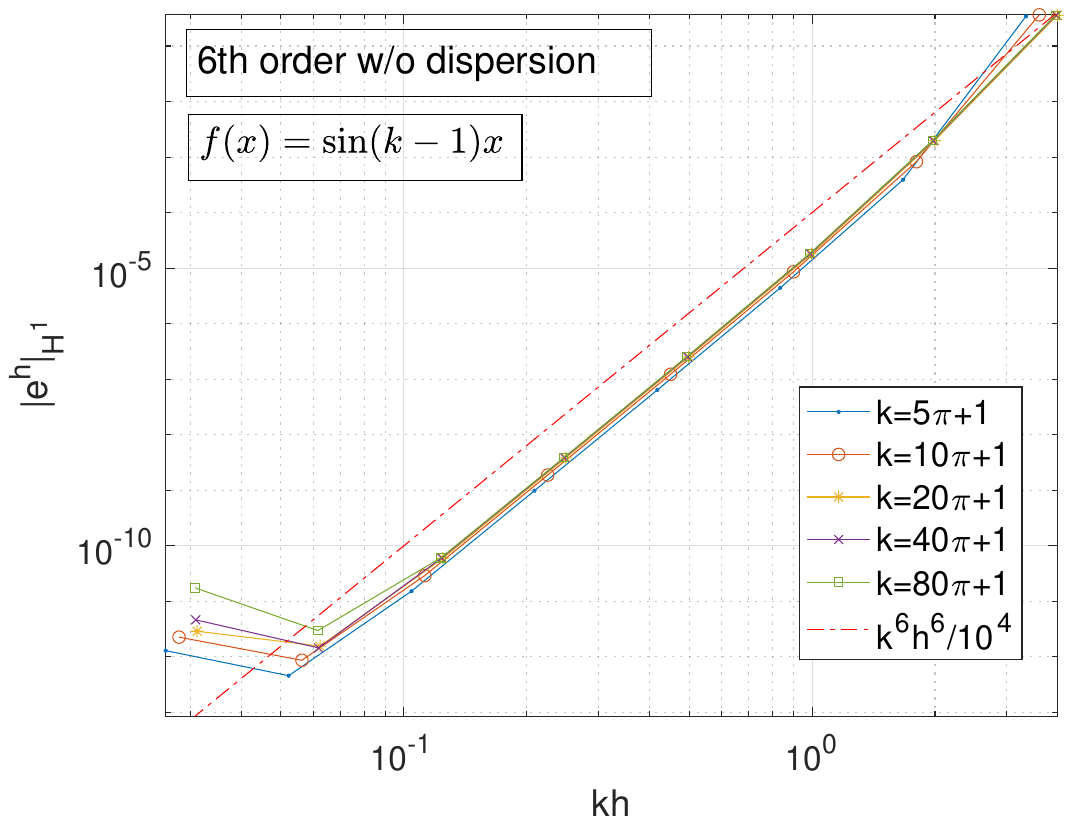}\;%
  \includegraphics[height=14em, width=12em,trim=30 180 75 180,clip]{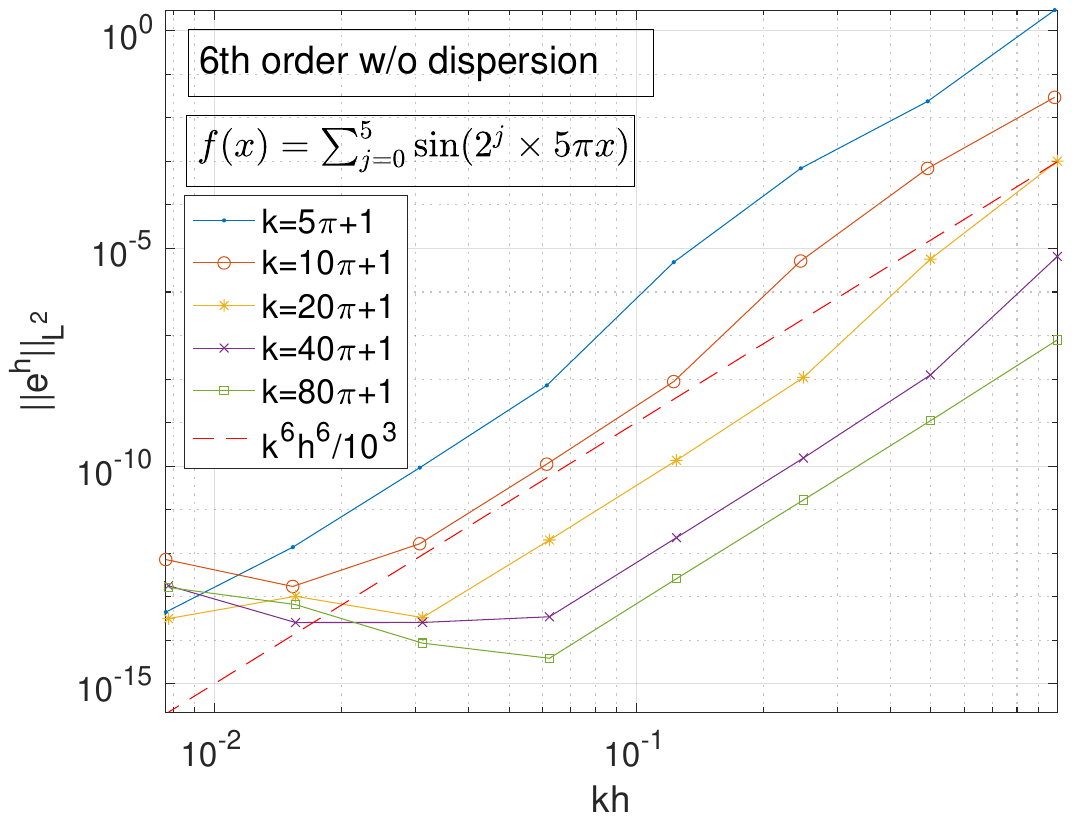}\;%
  \includegraphics[height=14em, width=12em,trim=30 180 75 180,clip]{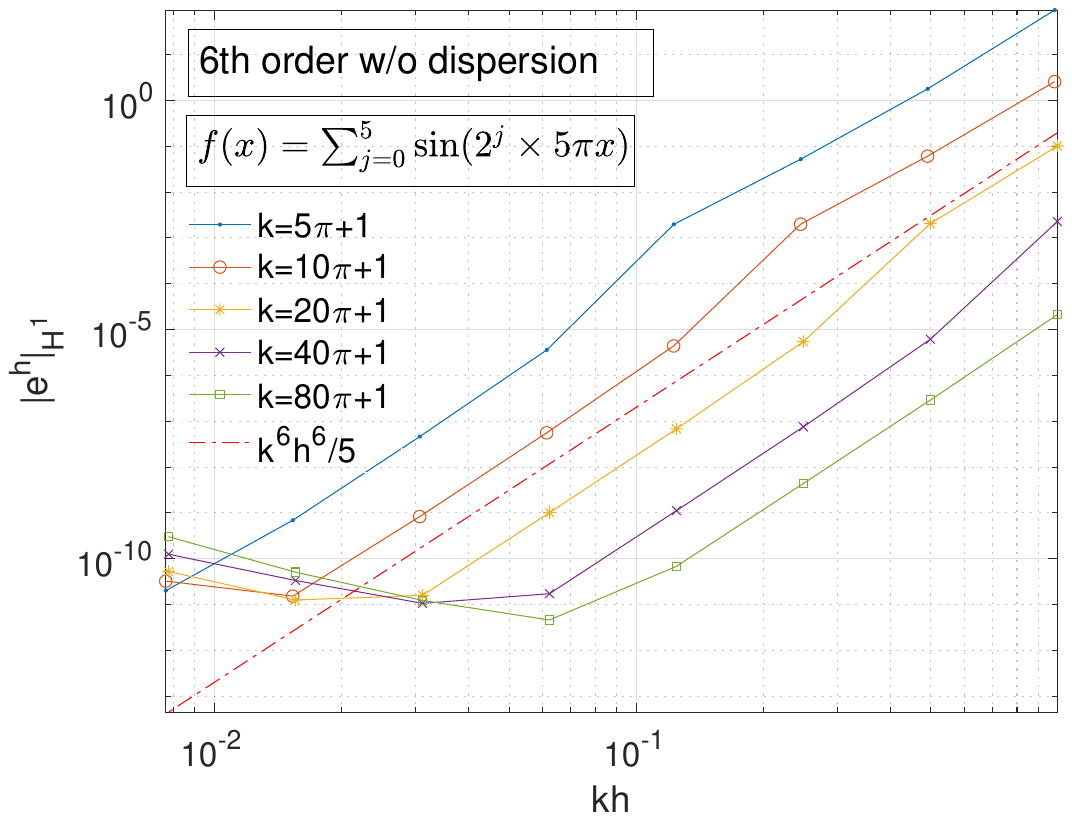}
  \caption{Numerical errors in 1D: monochromatic $k$-dependent source (col.1-2), mixed and fixed
    source (col.3-4).}\label{1dnum}
\end{figure}



\bibliographystyle{plain}
\bibliography{fdmanalref.bib}

\begin{thebibliography}{1}

\bibitem{babuska1997pollution}
Ivo~M Babuska and Stefan~A Sauter.
\newblock Is the pollution effect of the {FEM} avoidable for the {H}elmholtz
  equation considering high wave numbers?
\newblock {\em SIAM Journal on Numerical Analysis}, 34(6):2392--2423, 1997.

\bibitem{cocquet2021closed}
Pierre-Henri Cocquet, Martin~J Gander, and Xueshuang Xiang.
\newblock Closed form dispersion corrections including a real shifted
  wavenumber for finite difference discretizations of 2D constant coefficient
  {H}elmholtz problems.
\newblock {\em SIAM Journal on Scientific Computing}, 43(1):A278--A308, 2021.

\bibitem{deraemaeker1999dispersion}
Arnaud Deraemaeker, Ivo Babu{\v{s}}ka, and Philippe Bouillard.
\newblock Dispersion and pollution of the {FEM} solution for the {H}elmholtz
  equation in one, two and three dimensions.
\newblock {\em International Journal for Numerical Methods in Engineering},
  46(4):471--499, 1999.

\bibitem{dwarka2021pollution}
Vandana Dwarka and Cornelis Vuik.
\newblock Pollution and accuracy of solutions of the {H}elmholtz equation: A
  novel perspective from the eigenvalues.
\newblock {\em Journal of Computational and Applied Mathematics}, 395:113549,
  2021.

\bibitem{fu2008compact}
Yiping Fu.
\newblock Compact fourth-order finite difference schemes for {H}elmholtz
  equation with high wave numbers.
\newblock {\em Journal of Computational Mathematics}, pages 98--111, 2008.

\bibitem{stolk2016dispersion}
Christiaan~C Stolk.
\newblock A dispersion minimizing scheme for the 3-{D} {H}elmholtz equation
  based on ray theory.
\newblock {\em Journal of Computational Physics}, 314:618--646, 2016.

\bibitem{wang2014pollution}
Kun Wang and Yau~Shu Wong.
\newblock Pollution-free finite difference schemes for non-homogeneous
  {H}elmholtz equation.
\newblock {\em International Journal of Numerical Analysis and Modeling},
  11(4), 2014.

\bibitem{wu2017dispersion}
Tingting Wu.
\newblock A dispersion minimizing compact finite difference scheme for the 2{D}
  {H}elmholtz equation.
\newblock {\em Journal of Computational and Applied Mathematics}, 311:497--512,
  2017.

\end{thebibliography}

\end{document}